\documentclass[opre,nonblindrev]{informs3}

\DoubleSpacedXI 

\usepackage{endnotes}
\let\footnote=\endnote

\usepackage{bm}
\usepackage{caption}
\usepackage{subcaption}
\usepackage{rotating}
\usepackage{mathtools}
\usepackage{pgfplots}
\usepackage{algorithm}
\usepackage{algpseudocode}
\usepackage{booktabs}
\usepackage{cases}
\usepackage{enumerate}
\usepackage{hyperref}
\usepackage[normalem]{ulem}

\pgfplotsset{compat=1.17}
\newcolumntype{C}[1]{>{\centering\let\newline\\\arraybackslash\hspace{0pt}}m{#1}}

\usepackage{xcolor}
\usepackage{tikz}
\usetikzlibrary{arrows}
\usetikzlibrary{shapes}
\definecolor{tempcolor}{RGB}{244,244,244}
\usetikzlibrary{positioning}
\tikzset{
    split/.style = {shape=rectangle,
                     draw, align=center,
                     fill=tempcolor},
    clust/.style = {align=center,
                    draw,
                    rectangle}
                     }

\tikzstyle{process} = [rectangle, minimum width=3cm, minimum height=1cm, text centered, draw=black, fill=orange!30]
\tikzstyle{arrow} = [thick,->,>=stealth]

\algblockdefx{FORP}{ENDFORP}[1]%
  {\textbf{for }#1 \textbf{do in parallel}}%
  {\textbf{end for}}

\usepackage{natbib}
 \bibpunct[, ]{(}{)}{,}{a}{}{,}%
 \def\bibfont{\small}%
\TheoremsNumberedThrough     
\ECRepeatTheorems

\EquationsNumberedThrough    

\MANUSCRIPTNO{OPRE-2021-11-707.R1} 

\begin{document}


\RUNAUTHOR{Maragno et al.} 

\RUNTITLE{Mixed-integer Optimization with Constraint Learning}

\TITLE{Mixed-Integer Optimization with \\ Constraint Learning}

\ARTICLEAUTHORS{%
\AUTHOR{Donato Maragno\footnote{These authors contributed equally.}}
\AFF{Amsterdam Business School, University of Amsterdam, 1018 TV Amsterdam,
Netherlands \EMAIL{d.maragno@uva.nl}}
\AUTHOR{Holly Wiberg\footnotemark[\value{footnote}]}
\AFF{Operations Research Center, Massachusetts Institute of Technology, Cambridge MA 02139 \EMAIL{hwiberg@mit.edu}}
\AUTHOR{Dimitris Bertsimas}
\AFF{Sloan School of Management, Massachusetts Institute of Technology, Cambridge MA 02139 \EMAIL{dbertsim@mit.edu}}
\AUTHOR{\c{S}. \.{I}lker Birbil, Dick den Hertog, Adejuyigbe O. Fajemisin}
\AFF{Amsterdam Business School, University of Amsterdam, 1018 TV Amsterdam,
Netherlands \\ \EMAIL{s.i.birbil@uva.nl} \EMAIL{d.denhertog@uva.nl} \EMAIL{a.o.fajemisin2@uva.nl}}
} 
\ABSTRACT{%
We establish a broad methodological foundation for mixed-integer optimization with learned constraints. We propose an end-to-end pipeline for data-driven decision making in which constraints and objectives are directly learned from data using machine learning, and the trained models are embedded in an optimization formulation. We exploit the mixed-integer optimization-representability of many machine learning methods, including linear models, decision trees, ensembles, and multi-layer perceptrons,
which allows us to capture various underlying relationships between decisions, contextual variables, and outcomes.  We also introduce two approaches for handling the inherent uncertainty of learning from data. First, we characterize a decision trust region using the convex hull of the observations, to ensure credible recommendations and avoid extrapolation. We efficiently incorporate this representation using column generation and propose a more flexible formulation to deal with low-density regions and high-dimensional datasets. Then, we propose an ensemble learning approach that enforces constraint satisfaction over 
multiple bootstrapped estimators or multiple algorithms. In combination with domain-driven components, 
the embedded models and trust region define a mixed-integer optimization problem for prescription generation. We implement this framework as a Python package (\texttt{OptiCL}) for practitioners. We demonstrate the method in both World Food Programme planning and chemotherapy optimization. The case studies illustrate the framework's ability to generate high-quality prescriptions as well as the value added by the trust region, the use of ensembles to control model robustness, the consideration of multiple machine learning methods, and the inclusion of multiple learned constraints.
}%

\KEYWORDS{mixed-integer optimization, machine learning, constraint learning, prescriptive analytics} 

\maketitle

\section{Introduction}
Mixed-integer optimization (MIO) is a powerful tool that allows us to optimize a given objective subject to various constraints. This general problem statement of optimizing under constraints is nearly universal in decision-making settings. Some problems have readily quantifiable and explicit objectives and constraints, in which case MIO can be directly applied. The situation becomes more complicated, however, when the constraints and/or objectives are not explicitly known. 


For example, suppose we deal with cancerous tumors and want to prescribe a treatment regimen with a limit on toxicity; we may have observational data on treatments and their toxicity outcomes, but we have no natural function that relates the treatment decision to its resultant toxicity. We may also encounter constraints that are not directly quantifiable. Consider a setting where we want to recommend a diet, defined by a combination of foods and quantities, that is sufficiently ``palatable." Palatability cannot be written as a function of the food choices, but we may have qualitative data on how well people ``like" various potential dietary prescriptions. In both of these examples, we cannot directly represent the outcomes of interest as functions of our decisions, but we have \textit{data} that relates the outcomes and decisions. This raises a question: how can we consider data to learn these functions?

In this work, we tackle the challenge of data-driven decision making through a combined machine learning (ML) and MIO approach. ML allows us to learn functions that relate decisions to outcomes of interest directly through data. Importantly, many popular ML methods result in functions that are MIO-representable, meaning that they can be embedded into MIO formulations. This MIO-representable class includes both linear and nonlinear models, allowing us to capture a broad set of underlying relationships in the data. While the idea of learning functions directly from data is core to the field of ML, data is often underutilized in MIO settings due to the need for functional relationships between decision variables and outcomes. We seek to bridge this gap through \textit{constraint learning}; we propose a general framework that allows us to learn constraints and objectives directly from data, using ML, and to optimize decisions accordingly, using MIO. Once the learned constraints have been incorporated into the larger MIO, we can solve the problem directly using off-the-shelf solvers.

The term \textit{constraint learning}, used several times throughout this work, captures both constraints and objective functions. We are fundamentally learning functions to relate our decision variables to the outcome(s) of interest. The predicted values can then either be incorporated as constraints or objective terms; the model learning and embedding procedures remain largely the same. For this reason, we refer to them both under the same umbrella of \textit{constraint learning}. We describe this further in Section~\ref{subsect:MIP_rep}.

\subsection{Literature review}
\label{subsect:litrev}

Previous work has demonstrated the use of various ML methods in MIO problems and their utility in different application domains. The simplest of these methods is the regression function, as the approach is easy to understand and easy to implement. Given a regression function learned from data, the process of incorporating it into an MIO model is straightforward, and the final model does not require complex reformulations. As an example, \cite{Bertsimas2016} use regression models and MIO to develop new chemotherapy regimens based on existing data from previous clinical trials. \cite{kleijnen2015design} provides further information on this subject.

More complex ML models have also been shown to be MIO-representable, although  more effort is required to represent them than simple regression models. Neural networks which use the ReLU activation function can be represented using binary variables and big-M formulations \citep{Amos2016,Grimstad2019,Anderson2020,Chen2020b,Spyros2020,Venzke2020}. Where other activation functions are used \citep{Gutierrez-Martinez2011,lombardi2017empirical,Schweidtmann2019}, the MIO representation of neural networks is still possible, provided the solvers are capable of handling these functions.

With decision trees, each path in the tree from root to leaf node can be represented using one or more constraints \citep{Bonfietti2015,Verwer2017,Halilbasic2018}. The number of constraints required to represent decision trees is a function of the  tree size, with larger trees requiring more linearizations and binary variables. The advantage here, however, is that decision trees are known to be highly interpretable, which is often a requirement of ML in critical application settings~\citep{Thams2017}. Random forests \citep{Biggs2017,Misic2020} and other tree ensembles \citep{Cremer2019} have also been used in MIO in the same way as decision trees, with one set of constraints for each tree in the forest/ensemble along with one or more additional aggregate constraints. 

Data for constraint learning can either contain information on continuous data, feasible and infeasible states (two-class data), or only one state (one-class data). The problem of learning functions from one-class data and embedding them into optimization models has been recently investigated with the use of decision trees \citep{Kuda2018}, genetic programming \citep{Pawlak2019a}, local search \citep{Sroka2018}, evolutionary strategies \citep{Pawlak2019b}, and a combination of clustering, principal component analysis and wrapping ellipsoids \citep{Pawlak2021}.

The above selected applications generally involve a single function to be learned and a fixed ML method for the model choice. \citet{Verwer2017} use two model classes (decision trees and linear models) in a specific auction design application, but in this case the models were determined a priori. Some authors have presented a more general framework of embedding learned ML models in optimization problems such as JANOS \citep{bergman2019janos} and EML \citep{lombardi2017empirical}, but in practice these works are restricted to limited problem structures and learned model classes. We take a broader perspective, proposing a comprehensive end-to-end pipeline that encompasses the full ML and optimization components of a data-driven decision making problem. In contrast to EML and JANOS, \texttt{OptiCL} supports a wider variety of predictive models --- neural networks (with ReLU), linear regression, logistic regression, decision trees, random forests, gradient boosted trees and linear support vector machines. \texttt{OptiCL} is also more flexible than JANOS, as it can handle predictive models as constraints, and it also incorporates new concepts to deal with uncertainty in the ML models.
A comparison of \texttt{OptiCL} against  JANOS and EML on two test problems is shown in Appendix \ref{appendix:opticlvjanos}.

Our work falls under the umbrella of prescriptive analytics. \citet{Bertsimas2020} and \citet{Elmachtoub2021} leverage ML model predictions as inputs into an optimization problem. Our approach is distinct from existing work in that we directly embed ML models rather than extracting predictions, allowing us to optimize our decisions over the model. In the broadest sense, our framework relates to work that jointly harnesses ML and MIO, an area that has garnered significant interest in recent years in both the optimization and machine learning communities~\citep{bengio2021}.


\subsection{Contributions}
Our work unifies several research areas in a comprehensive manner. Our key contributions are as follows: 
\begin{enumerate}
    \item We develop an end-to-end framework that takes data and directly implements model training, model selection, integration into a larger MIO, and ultimately optimization. We make this available as an open-source software, \texttt{OptiCL} (Optimization with Constraint Learning) to provide a practitioner-friendly tool for making better data-driven decisions. The code is available at \texttt{https://github.com/hwiberg/OptiCL}. The software encompasses the full ML and optimization pipeline with the goal of being accessible to end users as well as extensible by technical researchers. Our framework natively supports models for both regression and classification functions and handles constraint learning in cases with both one-class and two-class data. We implement a cross-validation procedure for function learning that selects from a broad set of model classes. We also implement the optimization procedure in the generic mathematical modeling library Pyomo, which supports various state-of-the-art solvers.
    We introduce two approaches for handling the inherent uncertainty when learning from data. First, we propose an ensemble learning approach that enforces constraint satisfaction over an ensemble of multiple bootstrapped estimators or multiple algorithms, yielding more robust solutions. This addresses a shortcoming of existing approaches to embedding trained ML models, which rely on a single point prediction: in the case of learned constraints, model misspecification can lead to infeasibility. 
    Additionally, we restrict solutions to lie within a trust region, defined as the domain of the training data, which leads to better performance of the learned constraints. We offer several improvements to a basic convex hull formulation, including a clustering heuristic and a column selection algorithm that significantly reduce computation time. We also propose an enlargement of the convex hull which allows for exploration of solutions outside of the observed bounds. Both the ensemble model wrapper and trust region enlargement are controlled by parameters that allow an end user to directly trade-off the conservativeness of the constraint satisfaction. 
    \item We demonstrate the power of our method in two real-world case studies, using data from the World Food Programme and chemotherapy clinical trials. We pose relevant questions in the respective areas and formalize them as constraint learning problems. We implement our framework and subsequently evaluate the quantitative performance and scalability of our methods in these settings. 
\end{enumerate}

\section{Embedding predictive models}
\label{sec:methodology}
Suppose we have data $\mathcal{D} = \{ (\bm{\bar{x}}_i, \bm{\bar{w}}_i, \bm{\bar{y}}_i)\}_{i=1}^N$, with observed treatment decisions $\bm{\bar{x}}_i$, contextual information $\bm{\bar{w}}_i$, and outcomes of interest $\bm{\bar{y}}_i$ for sample $i$. Following the guidelines proposed in \cite{fajemisin2021optimization}, we present a framework that, given data $\mathcal{D}$, learns functions for the outcomes of interest ($\bm{y}$) that are to be constrained or optimized. These learned representations can then be used to generate predictions for a new observation with context $\bm{w}$. Figure~\ref{fig:overview} outlines the complete pipeline, which is detailed in the sections below.

\begin{figure}
    \FIGURE{\includegraphics[width=0.9\textwidth]{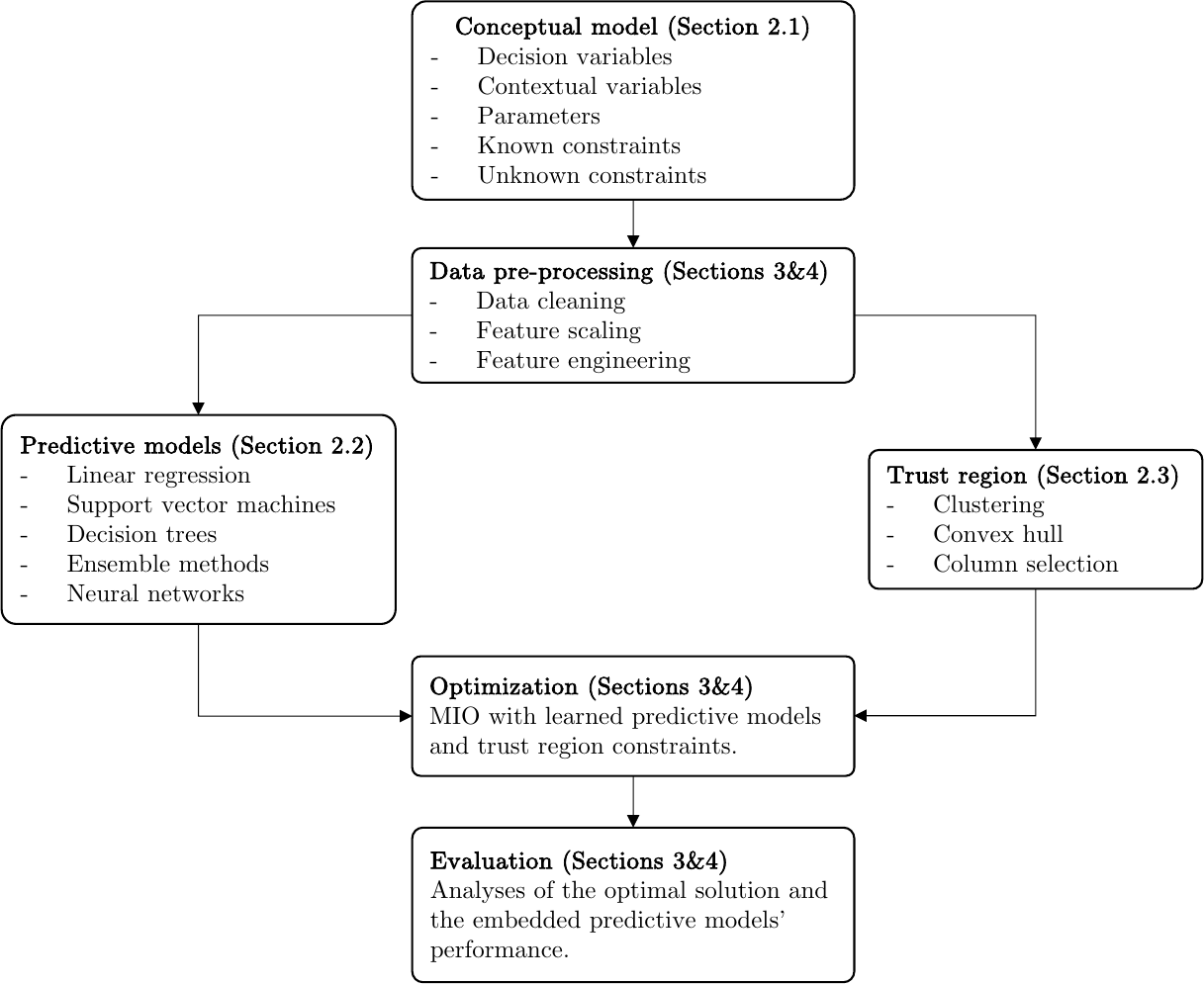}}
    {Constraint learning and optimization  pipeline.\label{fig:overview}}
    {}
\end{figure}

\subsection{Conceptual model}\label{subsect:conceptual}

Given the decision variable $\bm{x} \in \mathbb{R}^n$ and the fixed feature vector $\bm{w} \in \mathbb{R}^p$, we propose model \text{M}($\bm{w}$)
\begin{align}
\begin{aligned}
\label{eqn:conceptualmodel}
    \min_{\bm{x}\in \mathbb{R}^n,\bm{y}\in \mathbb{R}^k} \ & f(\bm{x}, \bm{w}, \bm{y}) & \\
    \mbox{s.t.} \ & \bm{g}(\bm{x}, \bm{w}, \bm{y}) \leq \bm{0}, & \\
    & \bm{y} = \hat{\bm{h}}_\mathcal{D}(\bm{x},\bm{w}), & \\
    & \bm{x} \in \mathcal{X}(\bm{w}), &
\end{aligned}
\end{align}
where $f(.,\bm{w},.):\mathbb{R}^{n+k} \mapsto \mathbb{R}$, $\bm{g}(.,\bm{w},.):\mathbb{R}^{n+k} \mapsto \mathbb{R}^m$, and $\hat{\bm{h}}_{\mathcal{D}}(.,\bm{w}):\mathbb{R}^{n} \mapsto \mathbb{R}^k$. Explicit forms of $f$ and $\bm{g}$ are known but they may still depend on the predicted outcome $\bm{y}$.
Here, $\hat{\bm{h}}_\mathcal{D}(\bm{x},\bm{w})$ represents the predictive models, one per outcome of interest, which are ML models trained on $\mathcal{D}$. Although our subsequent discussion mainly revolves around linear functions, we acknowledge the significant progress in nonlinear (convex) integer solvers. Our discussion can be easily extended to nonlinear models that can be tackled by those ever-improving solvers.


We note that the embedding of a single learned outcome may require multiple constraints and auxiliary variables; the embedding formulations are described in Section~\ref{subsect:MIP_rep}. For simplicity, we omit $\mathcal{D}$ in further notation of $\hat{\bm{h}}$ but note that all references to $\hat{\bm{h}}$ implicitly depend on the data used to train the model. Finally, the set $\mathcal{X}(\bm{w})$ defines the trust region, \textit{i.e.}, the set of solutions for which we trust the embedded predictive models. In Section \ref{subsect:convex_hull_as_trust_region}, we provide a detailed description of how the trust region $\mathcal{X}(\bm{w})$ is obtained from the observed data. We refer to the final MIO formulation with the embedded constraints and variables as EM($\bm{w}$). 

Model M($\bm{w}$) is quite general and encompasses several important \textit{constraint learning} classes:
\begin{enumerate}
    \item \textbf{Regression.} When the trained model results from a regression problem, it can be constrained by a specified upper bound $\tau$, \textit{i.e.}, $g(y) = y - \tau \leq 0$, or lower bound $\tau$, \textit{i.e.}, $g(y) = - y + \tau \leq 0$. If $\bm{y}$ is a vector (\textit{i.e.}, multi-output regression), we can likewise provide a threshold vector $\bm{\tau}$ for the constraints.
    \item \textbf{Classification.} If the trained model is obtained with a binary classification algorithm, in which the data is labeled as ``feasible" (1) or ``infeasible" (0), then the prediction is generally a probability $y \in [0,1]$. We can enforce a lower bound on the feasibility probability, \textit{i.e.}, $y \geq \tau$. A natural choice of $\tau$ is 0.5, which can be interpreted as enforcing that the result is more likely feasible than not. This can also extend to the multi-class setting, say $k$ classes, in which the output $\bm{y}$ is a $k$-dimensional unit vector, and we apply the constraint $y_i \geq \tau$ for whichever class $i$ is desired. When multiple classes are considered to be feasible, we can add binary variables to ensure that a solution is feasible, only if it falls in one of these classes with sufficiently high probability.
    \item \textbf{Objective function.} If the objective function has a term that is also learned by training an ML model, then we can introduce an auxiliary variable $t \in \mathbb{R}$, and add it to the objective function along with an epigraph constraint. Suppose for simplicity that the model involves a single learned objective function, $\hat{h}$, and no learned constraints. Then the general model becomes
    \begin{align*}
    \min_{\bm{x}\in \mathbb{R}^n,y\in \mathbb{R}, t\in\mathbb{R}} \ & t & \\
    \mbox{s.t.} \ & \bm{g}(\bm{x},\bm{w}) \leq 0, & \\
    & y = \hat{h}(\bm{x},\bm{w}), & \\
    & y - t \leq 0, & \\
    & \bm{x} \in \mathcal{X}(\bm{w}). &
    \end{align*}
    Although we have rewritten the problem to show the generality of our model, it is quite common in practice to use $y$ in the objective and omit the auxiliary variable $t$. 
\end{enumerate}

We observe that constraints on learned outcomes can be applied in two ways depending on the model training approach. Suppose that we have a continuous scalar outcome $y$ to learn and we want to impose an upper bound of $\tau \in \mathbb{R}$ (it may also be a lower bound without loss of generality). The first approach is called \textit{function learning} and concerns all cases where we learn a regression function $\hat{h} (\bm{x},\bm{w})$ without considering the feasibility threshold ($\tau$). The resultant model returns a predicted value $y \in \mathbb{R}$. The threshold is then applied as a constraint in the optimization model as $y \leq \tau$. Alternatively, we could use the feasibility threshold $\tau$ to binarize the outcome of each sample in $\mathcal{D}$ into feasible and infeasible, that is $\bar{y}_i \coloneqq \mathbb{I}(\bar{y}_i\leq \tau), \ i=1,\dots,N$, where $\mathbb{I}$ stands for the indicator function. After this relabeling, we train a binary classification model $\hat{h} (\bm{x},\bm{w})$ that returns a probability $y \in [0,1]$. This approach, called \textit{indicator function learning}, does not require any further use of the feasibility threshold $\tau$ in the optimization model, since the predictive models directly encode feasibility.

The function learning approach is particularly useful when we are interested in varying the threshold $\tau$ as a model parameter. Additionally, if the fitting process is expensive and therefore difficult to perform multiple times, learning an indicator function for each potential $\tau$ might be infeasible. In contrast, the indicator function learning approach is necessary when the raw data contains binary labels rather than continuous outcomes, and thus we have no ability to select or vary $\tau$.

\subsection{MIO-representable predictive models}
\label{subsect:MIP_rep}
Our framework is enabled by the ability to embed learned predictive models into an MIO formulation with linear constraints. This is possible for many classes of ML models, ranging from linear models to ensembles, and from support vector machines to neural networks. In this section, we outline the embedding procedure for decision trees, tree ensembles, and neural networks to illustrate the approach. We include additional technical details and formulations for these methods, along with linear regression and support vector machines, in Appendix~\ref{appendix:embedding_ml}.

In all cases, the model has been \textit{pre-trained}; we embed the trained model $\hat{h}(\bm{x},\bm{w})$ into our larger MIO formulation to allow us to constrain or optimize the resultant predicted value. Consequently, the optimization model is not dependent on the complexity of the model training procedure, but solely the size of the final trained model. Without loss of generality, we assume that $y$ is one-dimensional; \textit{i.e.}, we are learning a single model, and this model returns a scalar, not a multi-output vector.

All of the methods below can be used to learn constraints that apply upper or lower bounds to $y$, or to learn $y$ that we incorporate as part of the objective. We present the model embedding procedure for both cases when $\hat{h}(\bm{x},\bm{w})$ is a continuous or a binary predictive model, where relevant. We assume that either regression or classification models can be used to learn feasibility constraints, as described in Section~\ref{subsect:conceptual}. 

\paragraph{Decision Trees.}
Decision trees partition observations into distinct \textit{leaves} through a series of \textit{feature splits}. These algorithms are popular in predictive tasks due to their natural interpretability and ability to capture nonlinear interactions among variables. \cite{Breiman1984} first introduced Classification and Regression Trees (CART), which constructs trees through parallel splits in the feature space. Decision tree algorithms have subsequently been adapted and extended. \cite{Bertsimas2017} propose an alternative decision tree algorithm, Optimal Classification Trees (and Optimal Regression Trees), that improves on the basic decision tree formulation through an optimization framework that approximates globally optimal trees. Optimal trees also support multi-feature splits, referred to as \textit{hyper-plane splits}, that allow for splits on a linear combination of features~\citep{bertsimas2017optimalBOOK}.

A generic decision tree of depth 2 is shown in Figure~\ref{fig:dt_example}. A split at node $i$ is described by an inequality $A_i^\top \bm{x} \leq b_i$. We assume that $A$ can have multiple non-zero elements, in which we have the hyper-plane split setting; if there is only one non-zero element, this creates a parallel (single feature) split. Each terminal node $j$ (\textit{i.e.}, leaf) yields a prediction ($p_j$) for its observations. In the case of regression, the prediction is the average value of the training observations in the leaf, and in binary classification, the prediction is the proportion of leaf members with the feasible class. Each leaf can be described as a polyhedron, namely a set of linear constraints that must be satisfied by all leaf members. For example, for node 3, we define $\mathcal{P}_3 = \left\{x : A_1^\top x \leq b_1, A_2^\top x \leq b_2 \right\}$.


\begin{figure}
    \FIGURE{\includegraphics[scale=1]{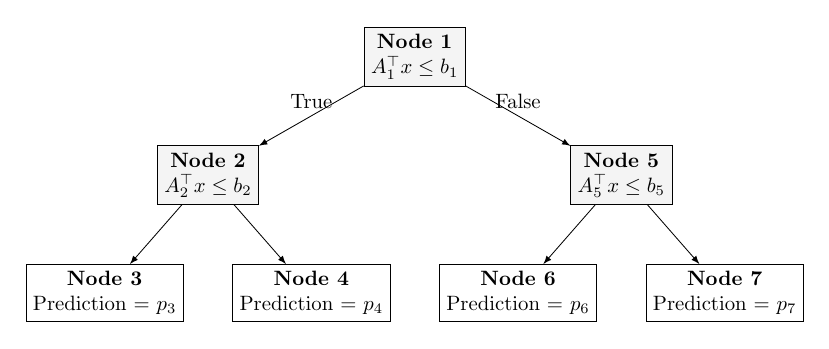}}
    {A decision tree of depth 2 with four terminal nodes (leaves).\label{fig:dt_example}}
    {}
\end{figure}

Suppose that we wish to constrain the predicted value of this tree to be at most $\tau$, a fixed constant. After obtaining the tree in Figure~\ref{fig:dt_example}, we can identify which paths satisfy the desired bound ($p_i \leq \tau$). Suppose that $p_3$ and $p_6$ do satisfy the bound, but $p_4$ and $p_7$ do not. In this case, we can enforce that our solution belongs to $\mathcal{P}_3$ or $\mathcal{P}_6$. This same approach applies if we only have access to two-class data (feasible vs. infeasible); we can directly train a binary classification algorithm and enforce that the solution lies within one of the ``feasible" prediction leaves (determined by a set probability threshold).

If the decision tree provides our only learned constraint, we can decompose the problem into multiple separate MIOs, one per feasible leaf. The conceptual model for the subproblem of leaf $i$ then becomes 
\begin{align*} 
    \min_{\bm{x}} \ & f(\bm{x},\bm{w}) \\
    \mbox{s.t.} \ & \bm{g}(\bm{x},\bm{w}) \leq 0, \\
    & (\bm{x},\bm{w}) \in \mathcal{P}_i,
\end{align*}
where the learned constraints for leaf $i$'s subproblem are implicitly represented by the polyhedron $\mathcal{P}_i$. These subproblems can be solved in parallel, and the minimum across all subproblems is obtained as the optimal solution. Furthermore, if all decision variables $\bm{x}$ are continuous, these subproblems are linear optimization problems (LOs), which can provide substantial computational gains. This is explored further in Appendix~\ref{appendix:dt}. 

In the more general setting where the decision tree forms one of many constraints, or we are interested in varying the $\tau$ limit within the model, we can directly embed the model into a larger MIO. We add binary variables representing each leaf, and set $y$ to the predicted value of the assigned leaf. An observation can only be assigned to a leaf, if it obeys all of its constraints; the structure of the tree guarantees that exactly one path will be fully satisfied, and thus, the leaf assignment is uniquely determined. A solution belonging to $\mathcal{P}_3$ will inherit $y = p_3$. Then, $y$ can be used in a constraint or objective. The full formulation for the embedded decision tree is included in Appendix~\ref{appendix:dt}. This formulation is similar to the proposal in~\citet{Verwer2017}. Both approaches have their own merits: while the Verwer formulation includes fewer constraints in the general case, our formulation is more efficient in the case where the problem can be decomposed into individual subproblems (as described above).

\paragraph{Ensemble Methods.}
Ensemble methods, such as random forests (RF) and gradient-boosting machines (GBM) consist of many decision trees that are aggregated to obtain a single prediction for a given observation. These models can thus be implemented by embedding many ``sub-models''~\citep{Breiman2001}. Suppose we have a forest with $P$ trees. Each tree can be embedded as a single decision tree (see previous paragraph) with the constraints from Appendix~\ref{appendix:dt}, which yields a predicted value $y_i$. 

RF models typically generate predictions by taking the average of the predictions from the individual trees: $$ y = \frac{1}{P}\sum_{i=1}^P y_i. $$
This can then be used as a term in the objective, or constrained by an upper bound as $y \leq \tau$; this can be done equivalently for a lower bound. In the classification setting, the prediction averages the probabilities returned by each model ($y_i \in [0,1]$), which can likewise be constrained or optimized.

Alternatively, we can further leverage the fact that unlike the other model classes, which return a single prediction, the RF model generates $P$ predictions, one per tree. We can impose a violation limit across the individual $P$ estimators as proposed in Section~\ref{subsect:robust_wrapper}. 

In the case of GBM, we have an ensemble of base-learners which are not necessarily decision trees. The model output is then computed as $$ y = \sum_{i=1}^P\beta_i y_i,$$ where $y_i$ is the predicted value of the $i$-th regression model $\hat{h}_i(\bm{x}, \bm{w})$, $\beta_i$ is the weight associated with the prediction. Although trees are typically used as base-learners, in theory we might use any of the MIO-representable predictive models discussed in this section.


\paragraph{Neural Networks.}
We implement multi-layer perceptrons (MLP) with a rectified linear unit (ReLU) activation function, which form an MIO-representable class of neural networks~\citep{Grimstad2019, Anderson2020}. These networks consist of an input layer, $L-2$ hidden layer(s), and an output layer. This nonlinear transformation of the input space over multiple nodes (and layers) using the ReLU operator ($v = \max\{0,x\}$) allows MLPs to capture complex functions that other algorithms cannot adequately encode, making them a powerful class of models.

Critically, the ReLU operator, $v = \max\{0,x\}$, can be encoded using linear constraints, as detailed in Appendix~\ref{appendix:mlp}. The constraints for an MLP network can be generated recursively starting from the input layer, which allows us to embed a trained MLP with an arbitrary number of hidden layers and nodes into an MIO. We refer to Appendix \ref{appendix:mlp} for details on the embedding of regression, binary classification, and multi-class classification MLP variants.

\subsection{Convex hull as trust region}
\label{subsect:convex_hull_as_trust_region}
As the optimal solutions of optimization problems are often at the extremes of the feasible region, this can be problematic for the validity of the trained ML model. Generally speaking the accuracy of a predictive model deteriorates for points that are further away from the data points in $\mathcal{D}$ \citep{goodfellow2015}. To mitigate this problem, we elaborate on the idea proposed by \cite{Biggs2017} to use the convex hull (CH) of the dataset as a trust region to prevent the predictive model from extrapolating.
According to \cite{Ebert2014}, when data is enclosed by a boundary of convex shape, the region inside this boundary is known as an interpolation region. This interpolation region is also referred to as the CH, and by excluding solutions outside the CH, we prevent extrapolation. If $\bm{X} = \{ \bm{\hat{x}}_i \}_{i=1}^N$ is the set of observed input data with $\bm{\hat{x}}_i = (\bm{\bar{x}}_i, \bm{\bar{w}}_i)$, we define the trust region as the CH of this set and denote it by CH($\bm{X}$). Recall that CH($\bm{X}$) is the smallest convex polytope that contains the set of points $\bm{X}$. It is well-known that computing the CH is exponential in time and space with respect to the number of samples and their dimensionality~\cite{Skiena_2008}. However, since the CH is a polytope, explicit expressions for its facets are not necessary. More precisely, CH($\bm{X}$) is represented as
\begin{align}
    \text{CH($\bm{X}$)} = \bigg\{ \bm{x} \bigg| \sum_{i \in \mathcal{I}} \lambda_i \bm{\hat{x}}_i = \bm{x}, \ \sum_{i \in \mathcal{I}} \lambda_i = 1, \ \bm{\lambda} \geq 0
    \bigg\},
\label{eqn:trust_region1}
\end{align}
where $\bm{\lambda} \in \mathbb{R}^N$, and $\mathcal{I} = \{1, \dots, N \}$ is the index set of samples in 
$\bm{X}$.

In situations such as the one shown in Figure \ref{fig:cluster1}, CH($\bm{X}$) includes regions with few or no data points (low-density regions). Blindly using CH($\bm{X}$) in this case can be problematic if the solutions are found in the low-density regions. We therefore advocate the use of a two-step approach. First, clustering is used to identify distinct high-density regions, and then the trust region is represented as the union of the CHs of the individual clusters (Figure \ref{fig:cluster2}).

\begin{figure} 
\centering
\caption{Use of the two-step approach to remove low-density regions.}\label{fig:cluster12}
\begin{subfigure}[t]{0.45\textwidth}
        \centering
        \includegraphics{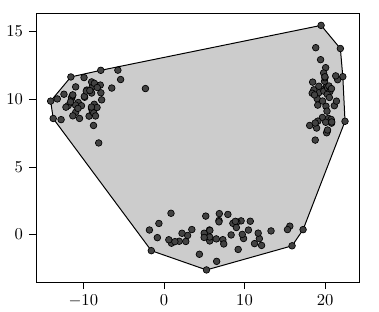}
        \caption{CH($\bm{X}$) with single region.}
        \label{fig:cluster1}
    \end{subfigure}
    \begin{subfigure}[t]{0.45\textwidth}
        \centering
        \includegraphics{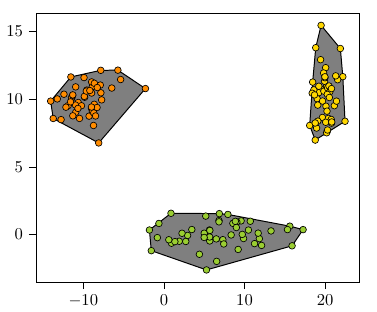}
        \caption{CH($\bm{X}$) with clustered regions.}
        \label{fig:cluster2}
    \end{subfigure}
    {}
\end{figure} 

\noindent We can either solve EM($\bm{w}$) for each  cluster, or embed the union of the $|\mathcal{K}|$ CHs into the MIO given by
\begin{align}
    \bigcup_{k\in\mathcal{K}}\text{CH($\bm{X}_k$)} = \bigg\{ \bm{x}
    \bigg| \sum_{i \in \mathcal{I}_k} 
    \lambda_i \bm{\hat{x}}_i = \bm{x}, \ \sum_{i \in \mathcal{I}_k}  \lambda_i = u_k \ \forall k \in \mathcal{K}, \sum_{k \in \mathcal{K}} u_k = 1, \ \bm{\lambda} \geq 0, \ \bm{u} \in \{0,1\}^{|\mathcal{K}|}
    \bigg\},
\label{eqn:trust_region2}
\end{align}
where $\bm{X}_k \subseteq \bm{X}$ refers to subset of samples in cluster $k \in \mathcal{K}$ with the index set $\mathcal{I}_k \subseteq \mathcal{I}$. The union of CHs requires the binary variables $u_k$ to constrain a feasible solution to be exactly in one of the CHs. More precisely, $u_k=1$ corresponds to the CH of the $k$-th cluster. As we show in Section \ref{subsect:wfp}, solving EM($\bm{w}$) for each  cluster may be done in parallel, which has a positive impact on computation time. We note that both formulations (\ref{eqn:trust_region1}) and (\ref{eqn:trust_region2}) assume that $\bm{\hat{x}}$ is continuous. These formulations can be extended to datasets with binary, categorical and ordinal features. In the case of categorical features, extra constraints on the domain and one-hot encoding are required.


Although the CH can be represented by linear constraints, the number of variables in EM($\bm{w}$) increases with the increase in the dataset size, which may make the optimization process prohibitive when the number of samples becomes too large. We therefore provide a column selection algorithm that selects a small subset of the samples. This algorithm can be directly used in the case of convex optimization problems or embedded as part of a branch and bound algorithm when the optimization problem involves integer variables. Figure~\ref{fig:column_selection} visually demonstrates the procedure; we begin with an arbitrary sample of the full data, and use column selection to iteratively add samples $\bm{\hat{x}}_i$ until no improvement can be found. In Appendix~\ref{appendix:columnselection}, we provide a full description of the approach, as well as a formal lemma which states that in each iteration of column selection, the selected sample from $\bm{X}$ is also a vertex of CH($\bm{X}$). In synthetic experiments, we observe that the algorithm scales well with the dataset size. The computation time required by solving the optimization problem with the algorithm is near-constant and minimally affected by the number of samples in the dataset. The experiments in Appendix~\ref{appendix:columnselection} show optimization with column selection to be significantly faster than a traditional approach, which makes it an ideal choice when dealing with massive datasets.

\begin{figure}
\FIGURE{\includegraphics[scale=0.5]{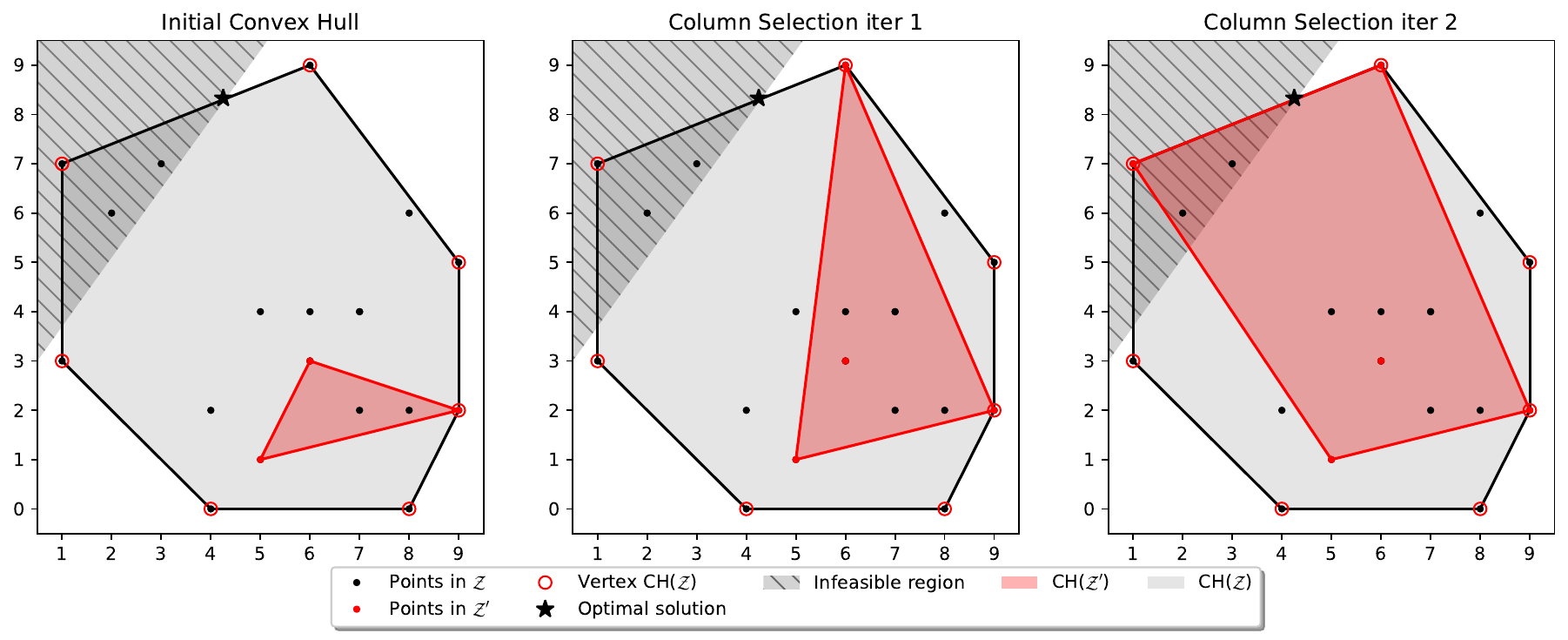}}{
Visualization of the column selection algorithm. Known and learned constraints define the infeasible region. The column selection algorithm starts using only a subset of data points (red filled circles), $\bm{X}^\prime\subseteq\bm{X}$ to define the trust region. In each iteration a vertex of CH($\bm{X}$) is selected (red hollow circle) and included in $\bm{X}^\prime$ until the optimal solution (star) is within the feasible region, namely the convex hull of $\bm{X}^\prime$. Note that with column selection we do not need the complete dataset to obtain the optimal solution, but rather only a subset.\label{fig:column_selection}}
{}
\end{figure}



\section{Uncertainty and Robustness}
\label{sec:robustness}
There are multiple sources of uncertainty, and consequently notions of robustness, that can be considered when embedding a trained machine learning model as a constraint. We define two types of uncertainty in model \eqref{eqn:conceptualmodel}.

\paragraph{Function Uncertainty.} The first source of uncertainty is in the underlying functional form of $\hat{h}$. We do not know the ground truth relationship between $(\bm{x},\bm{w})$ and $y$, and there is potential for model mis-specification. We mitigate this risk through our nonparametric model selection procedure, namely training $\hat{h}$ for a diverse set of methods (\textit{e.g.}, decision tree, regression, neural network) and selecting the final model using a cross-validation procedure.

\paragraph{Parameter Uncertainty.} Even within a single model class, there is uncertainty in the parameter estimates that define $\hat{h}$. Consider the case of linear regression. A regression estimator consists of point estimates of coefficients and an intercept term, but there is uncertainty in the estimates as they are derived from noisy data. We seek to make our model robust by characterizing this uncertainty and optimizing against it. We propose model-wrapper ensemble approaches, which are agnostic to the underlying model. The rest of this section addresses the model-wrapper approaches and a looser formulation of the trust region that prevents the optimal solution from being too conservative when the predictive models have good extrapolation performance.

\subsection{Model wrapper approach}\label{subsect:robust_wrapper}
We begin by describing the model ``wrapper" approach for characterizing uncertainty, in which we work directly with any trained models and their point predictions. Rather than obtaining our estimated outcome from a single trained predictive model, we suppose that we have $P$ estimators. The set of estimators can be obtained by bootstrapping or by training models using entirely different methods. The uncertainty is thus characterized by different realizations of the predicted value from multiple estimators, which effectively form an ensemble. 

We introduce a constraint that at most $\alpha \in [0,1]$ proportion of the $P$ estimators violate the constraint. Let $\hat{h}_1,\ldots,\hat{h}_P$ be the individual estimators. Then $\hat{h}_i(x) \leq \tau$ in at least $1-\alpha P$ of these estimators. This allows for a degree of robustness to individual model predictions by discarding a small number of potential outlier predictions. Formally,
\begin{align}
    \frac{1}{P}\sum_{i=1}^P \mathbb{I} (y_i \leq \tau) \geq 1 - \alpha. \label{eqn:model_wrapper}
\end{align}
Note that $\alpha = 0$ enforces the bound for all estimators, yielding the most conservative estimate, whereas $\alpha = 1$ removes the constraint entirely. Constraint \eqref{eqn:model_wrapper} is MIO-representable:
\begin{align*}
    & y_i \leq \tau + M(1-z_i), \ \ \ i=1,\dots, P \\
    & \frac{1}{P}\sum_{i=1}^P z_i  \geq 1-\alpha,
\end{align*}
where $z_i \in \{0, 1\} \ \forall i=1,\dots,P$, and $M$ is a sufficiently large constant. Appendix~\ref{app:model-wrapper} includes further details on this formulation and special cases. 

The violation limit concept can also be applied to estimators coming from multiple model classes, which allows us to enforce that the constraint is \textit{generally} obeyed when modeled through distinct methods. This provides a measure of robustness to \textit{function uncertainty}.

\subsection{Enlarged convex hull}
The use of the model wrapper approach and the trust region constraints, as defined in (\ref{eqn:trust_region1}), has a direct effect on the feasible region. The better performance of the learned constraints might be balanced out by the (potentially) unnecessary conservatism of the optimal solution. Although we introduced the trust region as a set of constraints to preserve the predictive performance of the fitted constraints, \cite{Balestriero2021} show how in a high-dimensional space the generalization performance of a fitted model is typically obtained extrapolating.
In light of this evidence, we propose an $\epsilon$-CH formulation which builds on (\ref{eqn:trust_region1}), and more generally on (\ref{eqn:trust_region2}). The relaxed formulation of the trust region enables the optimal solution of problem $M(\bm{w})$ to be outside $CH(X)$. 
Formally, we enlarge the trust region such that solutions outside CH($\bm{X}$) are considered feasible if they fall within the hyperball, with radius $\epsilon$, surrounding at least one of the data points in $\bm{X}$, see Figure~\ref{fig:eCH} (left). The $\epsilon$-CH is formulated as follows:
\begin{align}
\text{$\epsilon$-CH($\bm{X}$)} = \bigg\{ (\bm{x},\bm{s}) \bigg| \sum_{i \in \mathcal{I}} \lambda_i \bm{\hat{x}}_i = \bm{x} + \bm{s}, \ \sum_{i \in \mathcal{I}} \lambda_i = 1, \ \bm{\lambda} \geq 0, \ ||\bm{s}||_{p} \leq \epsilon
    \bigg\},
\label{eqn:epstrust_region1} 
\end{align}
with $\bm{s} \in \mathbb{R}^n$, and $p$ set equal to 1,2 or $\infty$ to preserve the complexity of the optimization problem. Figure~\ref{fig:eCH} (right) shows the extended region obtained with the $\epsilon$-CH. The choice of $\epsilon$ is pivotal in the trade-off between the performance of the learned constraints and the conservatism of the optimal solution. In the next section, we demonstrate how an increase in $\epsilon$ affects both the performance of the embedded predictive models and the objective function value.

\begin{figure*}[t!]
    \centering
    \caption{Trust region enlarged using an hyperball with radius $\epsilon$ around each sample in CH(X).}\label{fig:eCH}
    \begin{subfigure}[t]{0.5\textwidth}
        \centering
        \includegraphics[height=3.2in]{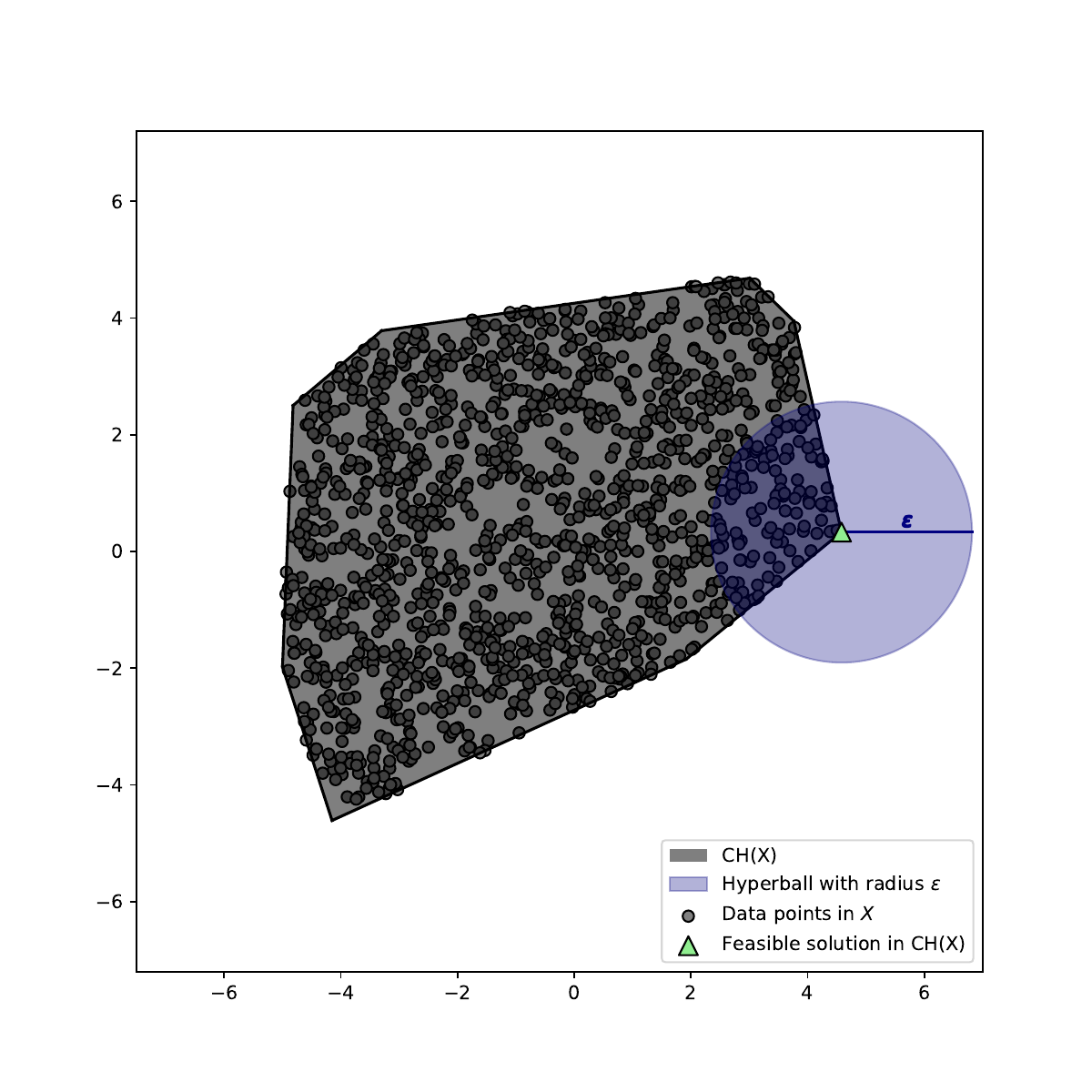}
    \end{subfigure}%
    ~ 
    \begin{subfigure}[t]{0.5\textwidth}
        \centering
        \includegraphics[height=3.2in]{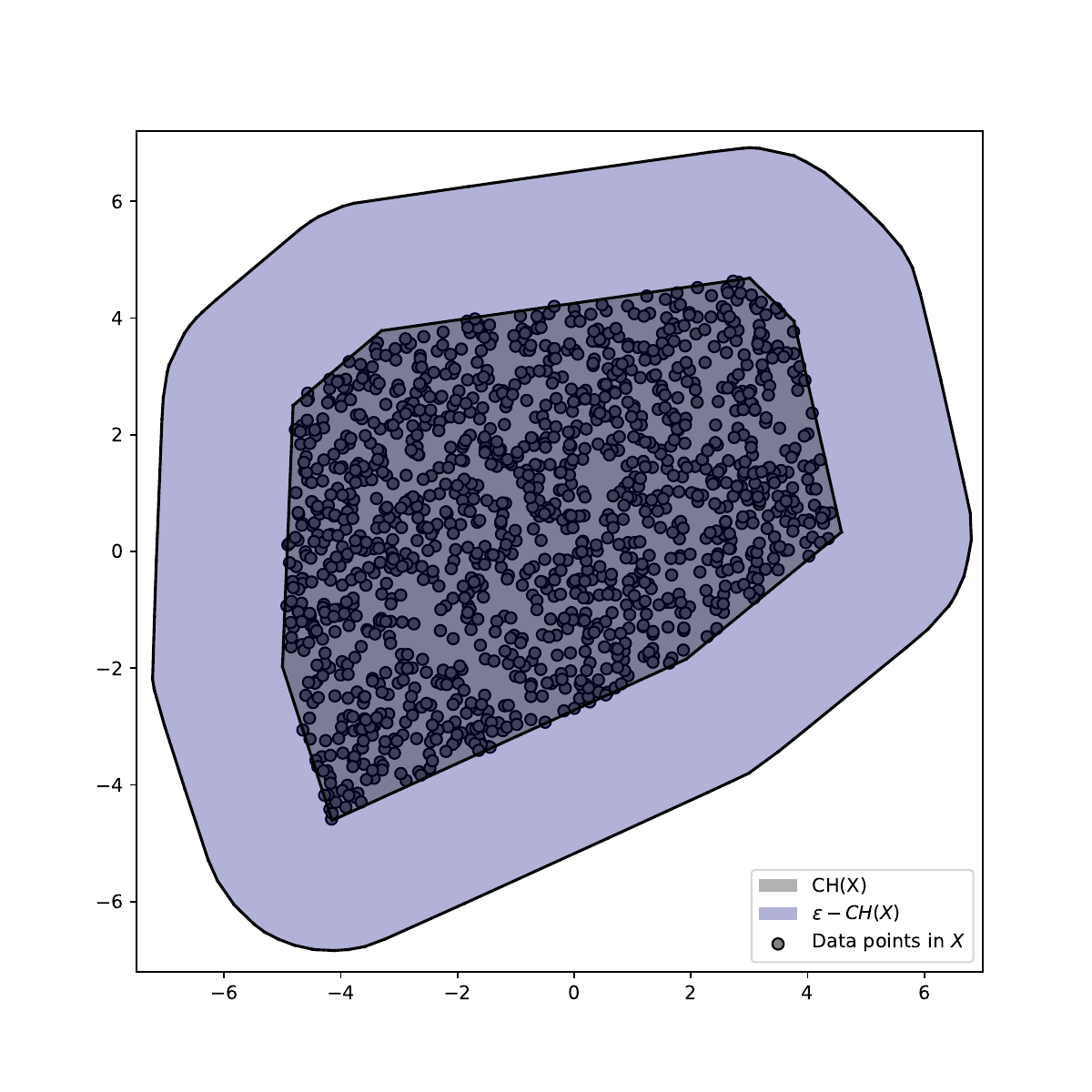}
    \end{subfigure}
\end{figure*}

\section{Case study: a palatable food basket for the World Food Programme}\label{subsect:wfp}
In this case study, we use a simplified version of the model proposed by \cite{peters2021nutritious}, which seeks to optimize humanitarian food aid. Its extended version aims to provide the World Food Programme (WFP) with a decision-making tool for long-term recovery operations, which simultaneously optimizes the food basket to be delivered, the sourcing plan, the delivery plan, and the transfer modality of a month-long food supply. The model proposed by \cite{peters2021nutritious} enforces that the food baskets address the nutrient gap and are palatable. To guarantee a certain level of palatability, the authors use a number of “unwritten rules” that have been defined in collaboration with nutrition experts. In this case study, we take a step further by inferring palatability constraints directly from data that reflects local people's opinions.
We use the specific case of Syria for this example. The conceptual model presents an LO structure with only the food palatability constraint to be learned. Data on palatability is generated through a simulator, but the procedure would remain unchanged if data were collected in the field, for example through surveys. The structure of this problem, which is an LO and involves only one learned constraint, allows the following analyses: (1) the effect of the trust-region on the optimal solution, and (2) the effect of clustering on the computation time and the optimal objective value. Additionally, the use of simulated data provides us with a ground truth to use in evaluating the quality of the prescriptions.
\subsection{Conceptual model}
\noindent
The optimization model is a combination of a capacitated, multi-commodity network flow model, and a diet model with constraints for nutrition levels and food basket palatability. 

The sets used to define the constraints and the objective function are displayed in Table \ref{tab:WFP_sets}. We have three different sets of nodes, and the set of commodities contains all the foods available for procurement during the food aid operation.
\begin{table}[htbp]
\TABLE{Definition of the sets used in the WFP model. \label{tab:WFP_sets}}{
\begin{tabular}{ll}
\toprule
\multicolumn{2}{c}{\textbf{Sets}}                          \\ \midrule
$\mathcal{N_S}$ & Set of source nodes                      \\
$\mathcal{N_T}$ & Set of transshipment nodes               \\
$\mathcal{N_D}$ & Set of delivery nodes                    \\
$\mathcal{K}$   & Set of commodities ($k \in \mathcal{K}$) \\
$\mathcal{L}$   & Set of nutrients ($l \in \mathcal{L}$)   \\ \bottomrule
\end{tabular}
}{}
\end{table}

The parameters used in the model are displayed in Table~\ref{tab:WFP_params}. The costs used in the objective function concern transportation ($p^T$) and procurement ($p^P$). The amount of food to deliver depends on the demand ($d$) and the number of feeding days ($days$). The nutritional requirements ($nutreq$) and nutritional values ($nutrval$) are detailed in Appendix~\ref{appendix:wfp}. The parameter $\gamma$ is needed to convert the metric tons used in the supply chain constraints to the grams used in the nutritional constraints. The parameter $t$ is used as a lower bound on the food basket palatability. The values of these parameters are based on those used by \cite{peters2021nutritious}.
\begin{table}[htbp]
\TABLE{Definition of the parameters used in the WFP model. \label{tab:WFP_params}}{
\begin{tabular}{ll}
\toprule
\multicolumn{2}{c}{\textbf{Parameters}}  \\ \midrule
$\gamma$      & Conversion rate from metric tons (mt) to grams (g) \\
$d_i$ & Number of beneficiaries at delivery point $i \in \mathcal{N_D}$ \\
$days$        & Number of feeding days \\
$nutreq_{l}$       & Nutritional requirement for nutrient $l \in \mathcal{L}$ (grams/person/day) \\
$nutval_{kl}$      & Nutritional value for nutrient $l \in \mathcal{L}$ per gram of commodity $k \in \mathcal{K}$ \\
$p_{ik}^P$    & Procurement cost (in \$ / mt) of commodity $k$ from source $i \in \mathcal{N_S}$                                                                  \\
$p_{ijk}^T$   & Transportation cost (in \$ / mt) of commodity $k$ from node $i \in \mathcal{N_S}\cup\mathcal{N_T}$ to node $j \in \mathcal{N_T}\cup\mathcal{N_D}$ \\
$t$      & Palatability lower bound \\
\bottomrule
\end{tabular}
}{}
\end{table}

The decision variables are shown in Table~\ref{tab:WFP_vars}. The flow variables $F_{ijk}$ are defined as the metric tons of a commodity $k$ transported from node $i$ to $j$. The variable $x_k$ represents the average daily ration per beneficiary for commodity $k$. The variable $y$ refers to the palatability of the food basket.

\begin{table}[htbp]
\TABLE{Definition of the variables used in the WFP model. \label{tab:WFP_vars}}{
\begin{tabular}{ll}
\toprule
\multicolumn{2}{c}{\textbf{Variables}}                                                               \\ \midrule
$F_{ijk}$ & Metric tons of commodity $k \in \mathcal{K}$ transported between node $i$ and node $j$ \\
$x_{k}$   & Grams of commodity $k \in \mathcal{K}$ in the food basket                         \\
$y$       & Food basket palatability                                                                 \\ \bottomrule
\end{tabular}
}{}
\end{table}

The full model formulation is as follows:
\begin{subequations}
\begin{align}
\min_{\bm{x}, y, \bm{F}} \ & \sum_{i \in \mathcal{N_S}} \sum_{j\in \mathcal{N_T \cup N_D}} \sum_{k \in \mathcal{K}} p_{ik}^PF_{ijk} + \sum_{i \in \mathcal{N_S \cup N_T}} \sum_{j \in \mathcal{N_T \cup N_D}} \sum_{k \in \mathcal{K}} p_{ijk}^TF_{ijk} & \label{eqn:WFPconstr0} \\
\mbox{s.t.} \ & \sum_{j \in \mathcal{N_T}} F_{ijk} = \sum_{j \in \mathcal{N_T}} F_{jik}, \ \ \ i \in \mathcal{N_T}, \  k \in \mathcal{K}, & \label{eqn:WFPconstr1}\\
& \sum_{j \in \mathcal{N_S \cup N_T}} \gamma F_{jik} = d_ix_kdays, \ \ \  i \in \mathcal{N_D}, \ k \in \mathcal{K}, & \label{eqn:WFPconstr2}\\
& \sum_{k \in \mathcal{K}} Nutval_{kl} x_{k} \geq Nutreq_{l}, \ \ \  l\in\mathcal{L}, & \label{eqn:WFPconstr4}\\
& x_{salt} = 5, & \label{eqn:WFPconstr8}\\
& x_{sugar} = 20, & \label{eqn:WFPconstr9}\\
& y \geq t, & \label{eqn:WFPconstr5}\\
& y = \hat{h}(\bm{x}), & \label{eqn:WFPconstr6}\\
& F_{ijk}, x_{k} \geq 0, \ \ \ i,j \in  \mathcal{N}, \ k \in \mathcal{K}. \label{eqn:WFPconstr7}
\end{align}
\end{subequations}
The objective function consists of two components, procurement costs and transportation costs. Constraints (\ref{eqn:WFPconstr1}) are used to balance the network flow, namely to ensure that the inflow and the outflow of a commodity are equal for each transhipment node. Constraints (\ref{eqn:WFPconstr2}) state that flow into a delivery node has to be equal to its demand, which is defined by the number of beneficiaries times the daily ration for commodity $k$ times the feeding days. Constraints (\ref{eqn:WFPconstr4}) guarantee an optimal solution that meets the nutrition requirements. Constraints (\ref{eqn:WFPconstr8}) and (\ref{eqn:WFPconstr9}) force the amount of salt and sugar to be 5 grams and 20 grams respectively. Constraint (\ref{eqn:WFPconstr5}) requires the food basket palatability ($y$), defined by means of a predictive model (\ref{eqn:WFPconstr6}), to be greater than a threshold ($t$). Lastly, non-negativity constraints (\ref{eqn:WFPconstr7}) are added for all commodity flows and commodity rations.

\subsection{Dataset and predictive models}
To evaluate the ability of our framework to learn and implement the palatability constraints, we use a simulator to generate diets with varying palatabilities. Each sample is defined by 25 features representing the amount (in grams) of all commodities that make up the food basket. We then use a ground truth function to assign each food basket a palatability between 0 and 1, where 1 corresponds to a perfectly palatable basket, and 0 to an inedible basket. This function is based on suggestions provided by WFP experts {\color{black} and complete details are outlined in Appendix~\ref{app:palatability}}. The data is then balanced to ensure that a wide variety of palatability scores are represented in the dataset. The final data used to learn the palatability constraint consists of 121,589 samples. Two examples of daily food baskets and their respective palatability scores are shown in Table~\ref{tab:foodbaskets}. In this case study, we use a palatability lower bound (t) of 0.5 for our learned constraint.

The next step of the framework involves training and choosing the predictive model that best approximates the unknown constraint. The predictive models used to learn the palatability constraints are those discussed in Section \ref{sec:methodology}, namely LR, SVM, CART, RF, GBM with decision trees as base-learners, and MLP with ReLU activation function.

\begin{table}[]
\TABLE{Two examples of daily food baskets. \label{tab:foodbaskets}}{
\begin{tabular}{@{}lcc@{}}
\toprule
\textbf{Commodity} & \textbf{Basket 1 Amount (g)} & \textbf{Basket 2 Amount (g)} \\ \midrule
DSM & 31.9 & 33.9 \\
Chickpeas & -- & 75.7 \\
Lentils & 41 & -- \\
Maize meal & 48.9 & -- \\
Meat & -- & 17.2 \\
Oil & 22 & 28.6 \\
Salt & 5 & 5 \\
Sugar & 20 & 20 \\
Wheat & 384.2 & 131.2 \\
Wheat flour & -- & 261.3 \\
WSB & 67.3 & 59.8 \\ \midrule
\textbf{Palatability Score} & 0.436 & 0.741 \\ \bottomrule
\end{tabular}
}{DSM=dried skim milk, WSB=wheat soya blend.}
\end{table}

\subsection{ Optimization results}
The experiments are executed using \texttt{OptiCL} jointly with Gurobi v9.1 \citep{gurobi} as the optimization solver. Table \ref{tab:WFP_results} reports the performances of the predictive models evaluated both for the validation set and for the prescriptions after being embedded into the optimization model. The table also compares the performance of the optimization with and without the trust region. The column ``Validation MSE" gives the Mean Squared Error (MSE) of each model obtained in cross-validation during model selection. While all scores in this column are desirably low, the MLP model significantly achieves the lowest error during this validation phase. The column ``MSE" gives the MSE of the predictive models once embedded into the optimization problem to evaluate how well the predictions for the optimal solutions match their true palatabilities (computed using the simulator). It is found using 100 optimal solutions of the optimization model generated with different cost vectors. The MLP model exhibits the best performance ($0.055$) in this context, showing its ability to model the palatability constraint better than all other methods.

\begin{table}[ht]
\TABLE
{Predictive models performances for the validation set (``Validation MSE"), and for the prescriptions after being embedded into the optimization model with (``MSE-TR") and without the trust region (``MSE"). The last two columns show the average computation time in seconds and its standard deviation (SD) required to solve the optimization model with (``Time-TR") and without the trust region (``Time").}
{\begin{tabular}{@{}lccccc@{}}
\toprule
\textbf{Model} & \textbf{Validation MSE} & \textbf{MSE} & \textbf{MSE-TR} & \textbf{Time (SD)} & \textbf{Time-TR (SD)}\\ \midrule
LR    & 0.046      & 0.256       & 0.042       & 0.003 (0.0008)  & 1.813 (0.204)  \\
SVM   & 0.019      & 0.226       & 0.027       & 0.003 (0.0006)  & 1.786 (0.208)  \\
CART  & 0.014      & 0.273       & 0.059       & 0.012 (0.0030)  & 7.495 (5.869)  \\
RF    & 0.018      & 0.252       & 0.025       & 0.248 (0.1050)  & 30.128 (13.917)   \\
GBM   & 0.006      & 0.250       & 0.017       & 0.513 (0.4562)  & 60.032 (41.685)   \\
MLP    & 0.001      & 0.055       & 0.001       & 14.905 (41.764)  & 28.405 (23.339)    \\ \bottomrule
\end{tabular}}
{Runtimes reported using an Intel i7-8665U 1.9 GHz CPU, 16 GB RAM (Windows 10 environment).\label{tab:WFP_results}}
\end{table}

\paragraph{Benefit of trust region.}
Table~\ref{tab:WFP_results} shows that when the trust region is used (``MSE-TR"), the MSEs obtained by all models are now much closer to the results from the validation phase. This shows the benefit of using the trust region as discussed in Section \ref{subsect:convex_hull_as_trust_region} to prevent extrapolation. With the trust region included, the MLP model also exhibits the lowest MSE ($0.001$). The improved performance seen with the inclusion of the trust region does come at the expense of computation speed. The column ``Time-TR" shows the average computation time in seconds and its standard deviation (SD) with trust region constraints included. In all cases, the computation time has clearly increased when compared against the computation time required without the trust region (column ``Time"). This is however acceptable, as significantly more accurate results are obtained with the trust region.

\paragraph{Benefit of clustering.}
The large dataset used in this case study makes the use of the trust region expensive in terms of time required to solve the final optimization model. While the column selection algorithm described in Section~\ref{subsect:convex_hull_as_trust_region} is ideal for significantly reducing the computation time, optimization models that require binary variables, either for embedding an ML model or to represent decision variables, would require column selection to be combined with a branch and bound algorithm. However, in this more general MIO case, it is possible to divide the dataset into clusters and solve in parallel an MIO for each cluster. By using parallelization, the total solution time can be expected to be equal to the longest time required to solve any single cluster's MIO. Contrary to column selection, the use of clusters can result in more conservative solutions; the trust region gets smaller with more clusters and prevents the model from finding solutions that are convex combinations of members of different clusters. However, as described in Section~\ref{subsect:convex_hull_as_trust_region}, solutions that lie between clusters may in fact reside in low-density areas of the feature space that should not be included in the trust region. In this sense, the loss in the objective value might actually coincide with more trustable solutions.

Figure \ref{fig:time_vs_obj} shows the effect of clusters in solving the model (\ref{eqn:WFPconstr0}-\ref{eqn:WFPconstr7}) with GBM as the predictive model used to learn the palatability constraint. K-means is used to partition the dataset into $K$ clusters, and the reported values are averaged over 100 iterations. In the left graph, we report the maximum runtime distribution across clusters needed to solve the different MIOs in parallel. In the right graph, we have the distributions of optimality gap, \textit{i.e.}, the relative difference between the optimal solution obtained with clusters compared to the solution obtained with no clustering. In this case study, the use of clusters significantly decreases the runtime (89.2\% speed up with $K=50$) while still obtaining near-optimal solutions (less then $0.25\%$ average gap with $K=50$). We observe that the trends are not necessarily monotonic in $K$. It is possible that a certain choice of $K$ may lead to a suboptimal solution, whereas a larger value of $K$ may preserve the optimal solution as the convex combination of points within a single cluster.

\begin{figure}[ht]
    \FIGURE
    {\includegraphics[width=\textwidth]{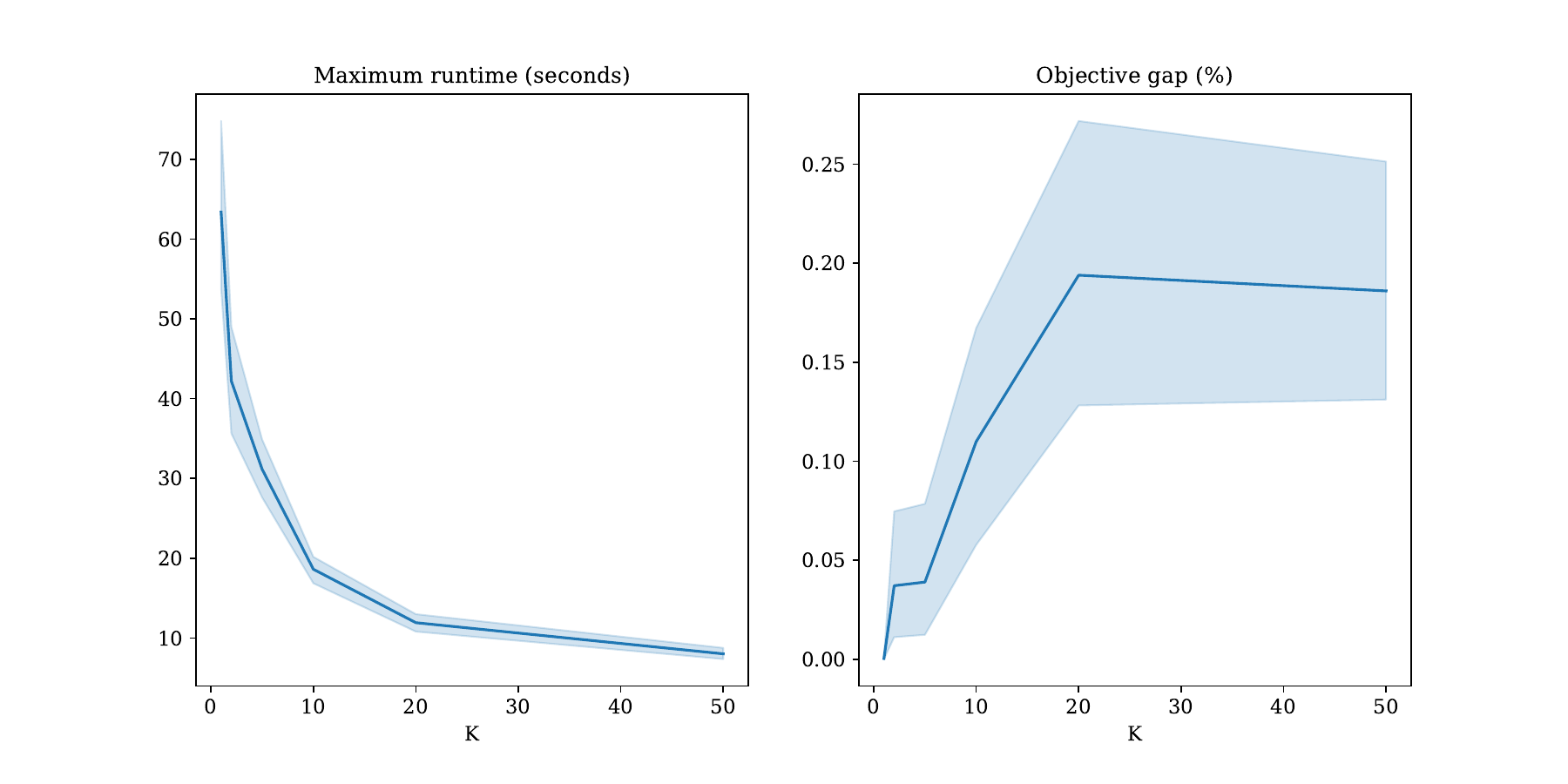}}
    {Effect of the number of clusters (K) on the computation time and the optimality gap across clusters, with bootstrapped 95\% confidence intervals.\label{fig:time_vs_obj}}
    {}
\end{figure}

\subsection{Robustness results} 
In these experiments, we assess the performance of the nominal and robust models. We consider three dimensions of performance: (1) true constraint satisfaction, (2) objective function value, and (3) runtime. The synthetic data used in this case study allows us to evaluate true palatability and constraint satisfaction as these parameters vary. This is the primary goal of the model wrapper ensemble approach, to improve feasibility and make solutions that are robust to any single learned estimator. 

We hypothesize that as our models become more conservative, we will more reliably satisfy the desired palatability constraint with some toll on the objective function. Additionally, embedding multiple models or characterizing uncertainty sets introduces computational complexity over a single nominal model. In this section, we compare the trade-offs in these metrics as we consider different notions of robustness and vary our conservativeness. We note that we are able to evaluate whether the true palatability meets the constraint threshold since palatability is defined through a known function. As with the experiments above, we solve the palatability problem with 100 different realizations of the cost vector and average the results.

The results below explore the effect of the $\alpha$ (violation limit) on cost and palatability in the WFP case study. Additional results on runtime, and experiments with varied estimators ($P$), are included in Appendix~\ref{appendix:robustness}. As the results demonstrate, the robustness parameters yield solutions that vary in their conservativeness and runtime. There is not a single set of optimal parameters. Rather, it is highly dependent on the use case, including factors like the stakes of the decision and the allowable turnaround time to generate solutions.

\paragraph{Multiple embedded models.} We first consider the impact of the model wrapper approach in the WFP problem. We compare different ways of embedding the palatability constraint, both using multiple estimators of a single model class and an ensemble containing multiple model classes. We run the experiments on a random sample of 1000 observations in the original WFP dataset. Within a single model class, we vary the number of estimators ($P \in [2,5,10,25]$) and the violation limit ($\alpha \in [0,0.1,0.2,0.5]$, or applying a mean constraint). Each estimator is obtained using a bootstrap sample (proportion = 0.5) of the underlying data. We compute metrics (1-3) for each variant to compare the tradeoffs in palatability (constraint satisfaction) and cost (objective function value). 

Figure~\ref{fig:cart_comparison} presents the results for a decision tree with $P=25$ and palatability threshold ($\tau$) equal to 0.5. The left figure shows the trade off between palatability and the objective as the violation limit ($\alpha$) varies. As expected, improvements in palatability (when $\alpha$ decreases) lead to increases in the total cost. However, we observe that a violation limit of 0.0 (vs. 0.5) leads to an 11.3\% improvement in real palatability (20.8\% improvement in predicted palatability), with a relatively modest 2.5\% increase in cost. The center and right figure show how palatability and violations vary with $\alpha$. Palatability increases and violations decease with lower $\alpha$. Both the violation rate (proportion of iterations with real palatability $< 0.5$) and violation margin (average distance to palatability threshold in cases where there is a violation) decrease with lower $\alpha$. This experiment demonstrates how the $\alpha$ parameter effectively controls the model's robustness as measured by constraint satisfaction. The approach has the advantage of parameterizing the violation limit, allowing us to explicitly control the model's conservativeness and evaluate constraint-objective tradeoffs.

\begin{figure}[ht]
    \FIGURE
    {\includegraphics[width=\linewidth]{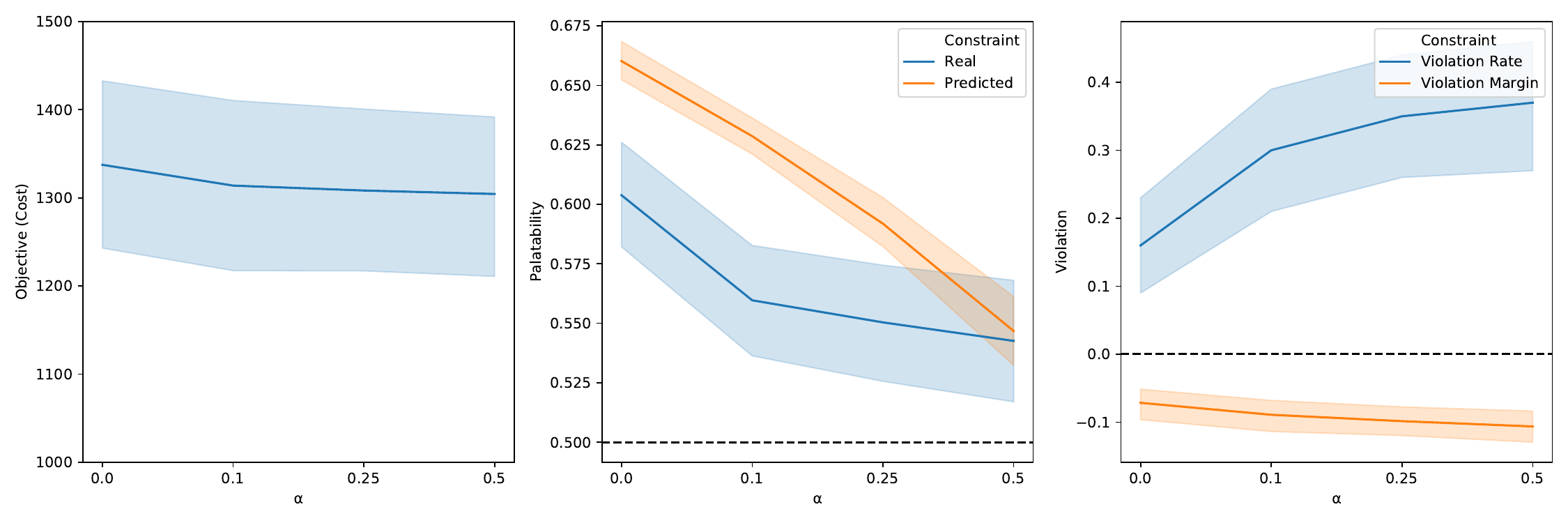}}
    {Comparison of CART models on objective function and constraint satisfaction.
    \label{fig:cart_comparison}}
    {}
\end{figure}

Appendix~\ref{appendix:robustness} reports further results for other model classes {\color{black} as well as runtime experiments}. 

\paragraph{Enlarged trust region.} In order to evaluate the effects of the enlarged trust region on the optimal solution, we use a simplified version of problem (\ref{eqn:WFPconstr0}-\ref{eqn:WFPconstr7}) where the only constraints are on the predictive model embedding, the palatability lower bound, and the $\epsilon$-CH. In Figure~\ref{fig:enlarged_TR}, we show how the objective function value and true palatability score vary according to different values of $\epsilon \in [0, 0.8]$. The results are obtained by averaging over 200 iterations with randomly generated cost vectors and using a decision tree as a predictive model to represent the palatability outcome. As expected, the objective value improves as $\epsilon$ increases. More interesting is the true palatability score which stays around the imposed lower bound of 0.5 for values of $\epsilon$ smaller than 0.25. This means that the predictive model is able to generalize even outside the CH as long as the optimal solution is not too far from it.

\begin{figure}[ht]
\FIGURE
{\includegraphics[width=0.7\textwidth]{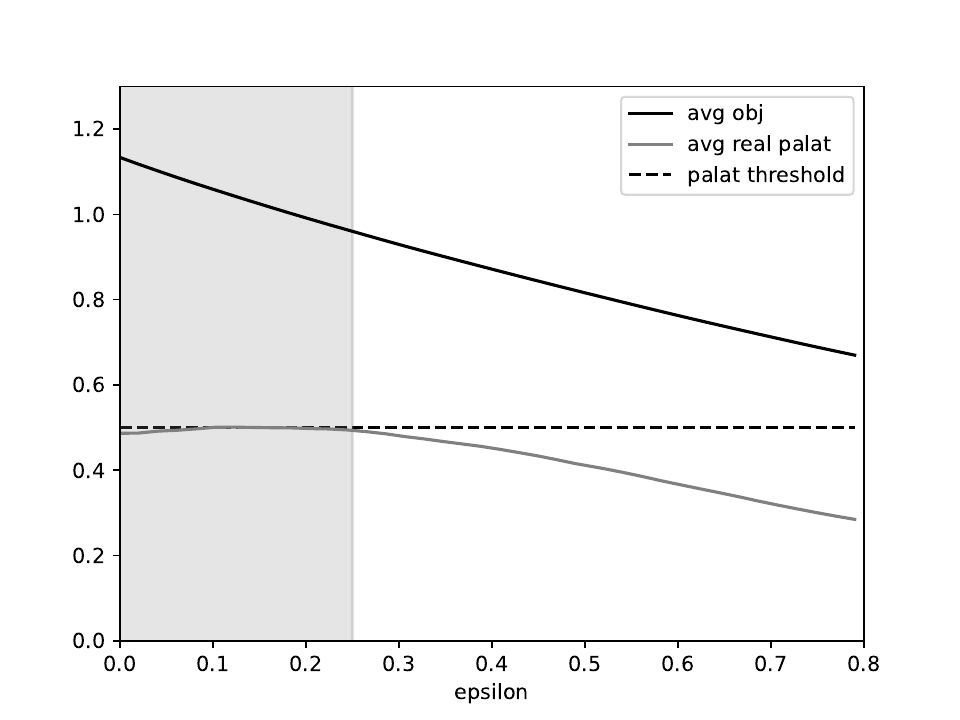}}
{Effect of the $\epsilon$-CH on the objective value and the predictive model performance with respect to the optimal solution. The values are obtained as an average of 200 iterations.\label{fig:enlarged_TR}}
{}
\end{figure}


\section{Case study: chemotherapy regimen design}\label{subsect:gastric}
In this case study, we extend the work of \citet{Bertsimas2016} in the design of chemotherapy regimens for advanced gastric cancer. Late stage gastric cancer has a poor prognosis with limited treatment options~\citep{Yang2011}. This has motivated significant research interest and clinical trials~\citep{nci_gi}. In \cite{Bertsimas2016}, the authors pose the question of algorithmically identifying promising chemotherapy regimens for new clinical trials based on existing trial results. They construct a database of clinical trial treatment arms which includes cohort and study characteristics, the prescribed chemotherapy regimen, and various outcomes. Given a new study cohort and study characteristics, they optimize a chemotherapy regimen to maximize the cohort's survival subject to a constraint on overall toxicity. The original work uses linear regression models to predict survival and toxicity, and it constrains a single toxicity measure. In this work we leverage a richer class of ML methods and more granular outcome measures. This offers benefits through higher performing predictive models and more clinically-relevant constraints.

Chemotherapy regimens are particularly challenging to optimize, since they involve multiple drugs given at potentially varying dosages, and they present risks for multiple adverse events that must be managed. This example highlights the generalizability of our framework to complex domains with multiple decisions and learned functions. The treatment variables in this problem consist of both binary and continuous elements, which are easily incorporated through our use of MIO. We have several learned constraints which must be simultaneously satisfied, and we also learn the objective function directly as a predictive model. 


\subsection{Conceptual model}\label{subsubsect:gastric:cm}
The use of clinical trial data forces us to consider each cohort as an observation, rather than an individual, since only aggregate measures are available. Thus, our model optimizes a cohort's treatment. The contextual variables ($\bm{w}$) consist of various cohort and study summary variables. The inclusion of fixed, \textit{i.e.}, non-optimization, features allows us to account for differences in baseline health status and risk across study cohorts. These features are included in the predictive models but then are fixed in the optimization model to reflect the group for whom we are generating a prescription. We assume that there are no unobserved confounding variables in this prescriptive setting.

The treatment variables ($\bm{x}$) encode a chemotherapy regimen. A regimen is defined by a set of drugs, each with an administration schedule of potentially varied dosages throughout a chemotherapy cycle. We characterize a regimen by drug indicators and each drug's average daily dose and maximum instantaneous dose in the cycle:
\begin{align*}
    & \bm{x}_b^d = \mathbb{I}(\text{drug $d$ is administered}), \\ 
    & \bm{x}_a^d = \text{average daily dose of drug $d$}, \\
    & \bm{x}_i^d = \text{maximum instantaneous dose of drug $d$}.
\end{align*}
This allows us to differentiate between low-intensity, high-frequency and high-intensity, low-frequency dosing strategies. The outcomes of interest ($\bm{y}$) consist of overall survival, to be included as the objective ($y_{OS}$), and various toxicities, to be included as constraints ($y_i, \ i \in \mathcal{Y}_C$).

To determine the optimal chemotherapy regimen $\bm{x}$ for a new study cohort with characteristics $\bm{w}$, we formulate the following MIO:
\begin{align*}
    \min_{\bm{x},\bm{y}} \ & y_{OS} & \\
    \mbox{s.t.} \ & y_i \leq \tau_i, &  {i \in \mathcal{Y}_C}, \\ 
    & y_i = \hat{h}_i(\bm{x},\bm{w}), &  i \in \mathcal{Y}_C, \\
    & y_{OS} = \hat{h}_{OS}(\bm{x},\bm{w}), \\
    & \sum_{d} \bm{x}_b^d \leq 3, \\ 
    & \bm{x}_b \in \{0,1\}^d, \\
     & \bm{x} \in \mathcal{X}(\bm{w}).
\end{align*}
    In this case study, we learn the full objective. However, this model could easily incorporate 
    deterministic components to optimize as additional weighted terms in the objective. 
    We include one domain-driven constraint, enforcing a maximum regimen combination of three drugs. 

The trust region, $\mathcal{X}(\bm{w})$, plays two crucial roles in the formulation. First, it ensures that the predictive models are applied within their valid bounds and not inappropriately extrapolated. It also naturally enforces a notion of ``clinically reasonable" treatments. It prevents drugs from being prescribed at doses outside of previously observed bounds, and it requires that the drug combination must have been previously seen (although potentially in different doses). It is nontrivial to explicitly characterize what constitutes a realistic treatment, and the convex hull provides a data-driven solution that integrates directly into the model framework. Furthermore, the convex hull implicitly enforces logical constraints between the different dimensions of $\bm{x}$. For example, a drug's average and instantaneous dose must be 0, if the drug's binary indicator is set to 0: this does not need to be explicitly included as a constraint, since this is true for all observed treatment regimens. The only explicit constraint required here is that the indicator variables $\bm{x}_b$ are binary.

\subsection{Dataset}\label{subsubsect:gastric:data}
Our data consists of 495 clinical trial arms from 1979-2012~\citep{Bertsimas2016}. We consider nine contextual variables, including the average patient age and breakdown of primary cancer site. There are 28 unique drugs that appear in multiple arms of the training set, yielding 84 decision variables. We include several ``dose-limiting toxicities" (DLTs) for our constraint set: Grade 3/4 constitutional toxicity, gastrointestinal toxicity, and infection, as well as Grade 4 blood toxicity. As the name suggests, these are chemotherapy side effects that are severe enough to affect the course of treatment. We also consider incidence of any dose-limiting toxicity (``Any DLT"), which aggregates over a superset of these DLTs.

We apply a temporal split, training the predictive models on trial arms through 2008 and generating prescriptions for the trial arms in 2009-2012. The final training set consists of 320 observations, and the final testing set consists of 96 observations. The full feature set, inclusion criteria, and data processing details are included in Appendix~\ref{appendix:gastric:data}. 

To define the trust region, we take the convex hull of the treatment variables ($\bm{x}$) on the training set. This aligns with the temporal split setting, in which we are generating prescriptions going forward based on an existing set of past treatment decisions. In general it is preferable to define the convex hull with respect to both $\bm{x}$ and $\bm{w}$ as discussed in Appendix~\ref{appendix:def_convex_hull}, but this does not apply well with a temporal split. Our data includes the study year as a feature to incorporate temporal effects, and so our test set observations will definitionally fall outside of the convex hull defined by the observed $(\bm{x},\bm{w})$ in our training set.

\subsection{Predictive models}\label{subsubsect:gastric:predictive}
Several ML models are trained for each outcome of interest using cross-validation for parameter tuning, and the best model is selected based on the validation criterion. We employ function learning for all toxicities, directly predicting the toxicity incidence and applying an upper bound threshold within the optimization model. 

Based on the model selection procedure, overall DLT, gastrointestinal toxicity, and overall survival are predicted using GBM models. Blood toxicity and infection are predicted using linear models, and constitutional toxicity is predicted with a RF model. This demonstrates the advantage of learning with multiple model classes; no single method dominates in predictive performance. A complete comparison of the considered models is included in Appendix~\ref{appendix:gastric:predictive}.

\subsection{Evaluation framework}
We generate prescriptions using the optimization model outlined in Section~\ref{subsubsect:gastric:cm}, with the embedded model choices specified in Section~\ref{subsubsect:gastric:predictive}. In order to evaluate the quality of our prescriptions, we must estimate the outcomes under various treatment alternatives. This evaluation task is notoriously challenging due to the lack of counterfactuals. In particular, we only know the true outcomes for observed cohort-treatment pairs and do not have information on potential unobserved combinations. We propose an evaluation scheme that leverages a ``ground truth" ensemble (GT ensemble). We train several ML models using all data from the study. These models are not embedded in an MIO model, so we are able to consider a broader set of methods in the ensemble. We then predict each outcome by averaging across all models in the ensemble. This approach allows us to capture the maximal knowledge scenario. Furthermore, such a ``consensus" approach of combining ML models has been shown to improve predictive performance and is more robust to individual model error~\citep{Bertsimas2021a_covidpresc}. The full details of the ensemble models and their predictive performances are included in Appendix~\ref{appendix:gastric:evaluation}. 


\subsection{Optimization results}\label{subsubsect:gastric:optimization_results}
We evaluate our model in multiple ways. We first consider the performance of our prescriptions against observed (given) treatments. We then explore the impact of learning multiple sub-constraints rather than a single aggregate toxicity constraint. All optimization models have the following shared parameters: toxicity upper bound of 0.6 quantile (as observed in training data) and maximum violation of 25\% for RF models. We report results for all test set observations with a feasible solution. It is possible that an observation has no feasible solution, implying that there is not a suitable drug combination lying within the convex hull for this cohort based on the toxicity requirements. These cases could be further investigated through a sensitivity analysis by relaxing the toxicity constraints or enlarging the trust region. With clinical guidance, one could evaluate the modifications required to make the solution feasible and the clinical appropriateness of such relaxations.

Table~\ref{tab:gastric:gt} reports the predicted outcomes under two constraint approaches: (1) constraining each toxicity separately (``All Constraints"), and (2) constraining a single aggregate toxicity measure (``DLT Only"). For each cohort in the test set, we generate predictions for all outcomes of interest under both prescription schemes and compute the relative change of our prescribed outcome from the given outcome predictions.

\paragraph{Benefit of prescriptive scheme.}
We begin by evaluating our proposed prescriptive scheme (``All Constraints") against the observed actual treatments. For example, under the GT ensemble scheme, 84.7\% of cohorts satisfied the overall DLT constraint under the given treatment, compared to 94.1\% under the proposed treatment. This yields an improvement of 11.10\%. We obtain a significant improvement in survival (11.40\%) while also improving toxicity limit satisfaction across all individual toxicities. Using the GT ensemble, we see toxicity satisfaction improvements between 1.3\%-25.0\%. We note that since toxicity violations are reported using the average incidence for each cohort, and the constraint limits are toxicity-specific, it is possible for a single DLT's incidence to be over the allowable limit while the overall ``Any DLT" rate is not.

\begin{table}[ht]
    \TABLE
    {Comparison of outcomes under given treatment regimen, regimen prescribed when only constraining the aggregate toxicity, and regimen prescribed under our full model.}
    {\begin{tabular}{@{}lccccc@{}}
\toprule
                 & \textbf{}  & \multicolumn{2}{c}{\textbf{All Constraints}}   & \multicolumn{2}{c}{\textbf{DLT Only}}  \\ \midrule
                 & \textbf{Given (SD)} & \textbf{Prescribed (SD)} & \textbf{\% Change} & \textbf{Prescribed (SD)} & \textbf{\% Change} \\ \midrule
Any DLT & 0.847 (0.362) & 0.941 (0.237) & 11.10\% & 0.906 (0.294) & 6.90\% \\
Blood & 0.812 (0.393) & 0.824 (0.383) & 1.40\% & 0.706 (0.458) & -13.00\% \\
Constitutional & 0.953 (0.213) & 1.000 (0.000) & 4.90\% & 1.000 (0.000) & 4.90\% \\
Infection & 0.882 (0.324) & 0.894 (0.310) & 1.30\% & 0.800 (0.402) & -9.30\% \\
Gastrointestinal & 0.800 (0.402) & 1.000 (0.000) & 25.00\% & 1.000 (0.000) & 25.00\% \\ \midrule
Overall Survival & 10.855 (1.939) & 12.092 (1.470) & 11.40\% & 12.468 (1.430) & 14.90\% \\ \bottomrule
\end{tabular}}
    {We report the mean and standard deviation (SD) of constraint satisfaction (binary indicator) and overall survival (months) across the test set. The relative change is reported against the given treatment.\label{tab:gastric:gt}}
\end{table}

\paragraph{Benefit of multiple constraints.}
Table~\ref{tab:gastric:gt} also illustrates the value of enforcing constraints on each individual toxicity rather than as a single measure. When only constraining the aggregate toxicity measure (``DLT Only"), the resultant prescriptions actually have lower constraint satisfaction for blood toxicity and infection than the baseline given regimens. By constraining multiple measures, we are able to improve across all individual toxicities. The fully constrained model actually improves the overall DLT measure satisfaction, suggesting that the inclusion of these      ``sub-constraints" also makes the aggregate constraint more robust. This improvement does come at the expense of slightly lower survival between the ``All" and ``DLT Only" models (-0.38 months) but we note that incurring the individual toxicities that are violated in the ``DLT Only" model would likely make the treatment unviable. 

\section{Discussion}
Our experimental results illustrate the benefits of our constraint learning framework in data-driven decision making in two problem settings: food basket recommendations for the WFP and chemotherapy regimens for advanced gastric cancer. The quantitative results show an improvement in predictive performance when incorporating the trust region and learning from multiple candidate model classes. 
Our framework scales to large problem sizes, enabled by efficient formulations and tailored approaches to specific problem structures. Our approach for efficiently learning the trust region also has broad applicability in one-class constraint learning. 

The nominal problem formulation is strengthened by embedding multiple models for a single constraint rather than relying on a single learned function. This notion of robustness is particularly important in the context of learning constraints: whereas mis-specfications in learned objective functions can lead to suboptimal outcomes, a mis-specified constraint can lead to infeasible solutions. Finally, our software exposes the model ensemble construction and trust region enlargement options directly through user-specified parameters. This allows an end user to directly evaluate tradeoffs in objective value and constraint satisfaction, as the problem's real-world context often shapes the level of desired conservatism.


We recognize several opportunities to further extend this framework. Our work naturally relates to the causal inference literature and individual treatment effect estimation~\citep{Athey2016,shalit2017estimating}. These methods do not directly translate to our problem setting; existing work generally assumes highly structured treatment alternatives (\textit{e.g.}, binary treatment vs. control) or a single continuous treatment (\textit{e.g.}, dosing), whereas we allow more general decision structures. In future work, we are interested in incorporating ideas from causal inference to relax the assumption of unobserved confounders.

Additionally, our framework is dependent on the quality of the underlying predictive models. We constrain and optimize point predictions from our embedded models. This can be problematic in the case of model misspecification, a known shortcoming of ``predict-then-optimize" methods~\citep{Elmachtoub2021}. We mitigate this concern in two ways. First, our model selection procedure allows us to obtain higher quality predictive models by capturing several possible functional relationships. Second, our model wrapper approach for embedding a single constraint with an ensemble of models allows us to directly control our robustness to the predictions of individual learners. In future work, there is an opportunity to  incorporate ideas from robust optimization to directly account for prediction uncertainty in individual model classes. While this has been addressed in the linear case~\citep{goldfarb2003}, it remains an open area of research in more general ML methods.

In this work, we present a unified framework for optimization with learned constraints that leverages both ML and MIO for data-driven decision making. Our work flexibly learns problem constraints and objectives with supervised learning, and incorporates them into a larger optimization problem of interest. We also learn the trust region, providing more credible recommendations and improving predictive performance, and accomplish this efficiently using column generation and unsupervised learning. The generality of our method allows us to tackle quite complex decision settings, such as chemotherapy optimization, but also includes tailored approaches for more efficiently solving specific problem types.
Finally, we implement this as a Python software package (\texttt{OptiCL}) to enable practitioner use. We envision that  \texttt{OptiCL}'s methodology will be added to state-of-the-art optimization modeling software packages. 

\ACKNOWLEDGMENT{The authors thank the anonymous reviewers and editorial team for their valuable feedback on this work. This work was supported by the Dutch Scientific Council (NWO) grant OCENW.GROOT.2019.015, Optimization for and with Machine Learning (OPTIMAL). Additionally, Holly Wiberg was supported by the National Science Foundation Graduate Research Fellowship under Grant No. 174530. Any opinion, findings, and conclusions or recommendations expressed in this material are those of the authors(s) and do not necessarily reflect the views of the National Science Foundation.}

\bibliographystyle{informs2014} 
\bibliography{References.bib}

\newpage \clearpage
\renewcommand{\theHchapter}{A\arabic{chapter}}
\renewcommand{\thetable}{EC.\arabic{table}}
\renewcommand{\thefigure}{EC.\arabic{figure}}
\setcounter{table}{0}
\setcounter{figure}{0}
\newpage
\begin{APPENDICES}

\section{Machine learning model embedding}\label{appendix:embedding_ml}


\subsection{Linear models}

\paragraph{Linear Regression.}
Linear regression (LR) is a natural choice of predictive function given its inherent linearity and ease of embedding. A regression model can be trained to predict the outcome of interest, $y$, as a function of $\bm{x}$ and $\bm{w}$. The algorithm can optionally use regularization; the embedding only requires the final coefficient vectors $\bm{\beta_x} \in \mathbb{R}^{n}$ and $\bm{\beta_w} \in \mathbb{R}^{p}$ (and intercept term $\beta_0$) to describe the model. The model can then be embedded as 
$$y = \beta_0 + \bm{\beta_x}^\top\bm{x} + \bm{\beta_w}^\top\bm{w}.$$

\paragraph{Support Vector Machines.}
A support vector machine (SVM) uses  a hyper-plane split to generate predictions, both for classification \citep{cortes1995support} and regression \citep{Drucker1997}. We consider the case of \textit{linear} SVMs, since this allows us to obtain the prediction as a linear function of the decision variables $\bm{x}$. In linear support vector regression (SVR), which we use for function learning, we fit a linear function to the data. The setting is similar to  linear regression, but the loss function only penalizes residuals greater than an $\epsilon$ threshold~\citep{Drucker1997}. 
As with linear regression, the trained model returns a linear function with coefficients $\bm{\beta_x}$, $\bm{\beta_w}$, and $\beta_0$. The final prediction is
$$y = \beta_0 + \bm{\beta_x}^\top\bm{x} + \bm{\beta_w}^\top\bm{w}.$$

\noindent For the classification setting, linear support vector classification (SVC) identifies a hyper-plane that best separates positive and negative samples~\citep{cortes1995support}. A trained SVC model similarly returns coefficients $\bm{\beta_x}$, $\bm{\beta_w}$, and $\beta_0$, where a sample's prediction is given by
\begin{align*}
    y = \begin{dcases}
    1, ~ & {\text{if }\beta_0 + \bm{\beta_x}^\top\bm{x} + \bm{\beta_w}^\top\bm{w} \geq 0};\\
    0, ~ & {\text{otherwise}}.
\end{dcases}
\end{align*}
In SVC, the output variable $y$ is binary rather than a probability. In this case, the constraint can simply be embedded as $\beta_0 + \bm{\beta_x}^\top\bm{x} + \bm{\beta_w}^\top\bm{w} \geq 0$.
\\

\subsection{Decision trees}\label{appendix:dt}

Consider the leaves in Figure~\ref{fig:dt_example}. An observation will be assigned to the leftmost leaf (node 3) if $A_1^\top \bm{x} \leq b_1$ and $A_2^\top \bm{x} \leq b_2$. An observation would be assigned to node 4 if  $A_1^\top \bm{x} \leq b_1$ and $A_2^\top \bm{x} > b_2$, or equivalently, $-A_2^\top \bm{x} < -b_2$. Furthermore, we can remove the strict inequalities using a sufficiently small $\epsilon$ parameter, so that $-A_2^\top \bm{x} \leq -b_2 - \epsilon$. We can then encode the leaf assignment of observation $\bm{x}$ through the following constraints:
\begin{subequations}
\begin{align}
A_1^\top \bm{x} - M(1-l_3) & \leq b_1,   \label{eqn:dt:node1_leaf3} \\ 
A_2^\top \bm{x} - M(1-l_3) &  \leq b_2, \label{eqn:dt:node2_leaf3} \\
A_1^\top \bm{x} - M(1-l_4) & \leq b_1, \label{eqn:dt:node1_leaf4} \\
- A_2^\top \bm{x} - M(1-l_4) & \leq -b_2 - \epsilon, \label{eqn:dt:node2_leaf4} \\
- A_1^\top \bm{x} - M(1-l_6) & \leq -b_1 - \epsilon, \label{eqn:dt:node1_leaf6} \\
A_5^\top \bm{x} - M(1-l_6) & \leq b_5, \label{eqn:dt:node5_leaf6} \\
- A_1^\top \bm{x} - M(1-l_7) & \leq - b_1 - \epsilon, \label{eqn:dt:node5_leaf7} \\
- A_5^\top \bm{x}  - M(1-l_7) & \leq - b_5 - \epsilon, \label{eqn:dt:node1_leaf7} \\ 
l_3 + l_4 + l_6 + l_7 & = 1, \label{eqn:dt:assignment} \\ 
y - (p_3 l_3 + p_4  l_4 + p_6 l_6 + p_7 l_7) &= 0, \label{eqn:dt:pred}
\end{align}
\end{subequations}
where $l_3,l_4,l_6,l_7$ are binary variables associated with the corresponding leaves. For a given $\bm{x}$, if $A_1^\top \bm{x} \leq b_1$, Constraints (\ref{eqn:dt:node1_leaf6}) and (\ref{eqn:dt:node1_leaf7}) will force $l_6$ and $l_7$ to zero, respectively. If $A_2^\top \bm{x} \leq b_2$, constraint (\ref{eqn:dt:node2_leaf4}) will force $l_4$ to 0. The assignment constraint (\ref{eqn:dt:assignment}) will then force $l_3 = 1$, assigning the observation to leaf 3 as desired. Finally, constraint (\ref{eqn:dt:pred}) sets $y$ to the prediction of the assigned leaf ($p_3$). We can then constrain the value of $y$ using our desired upper bound of $\tau$ (or lower bound, without loss of generality).

More generally, consider a decision tree $\hat{h}(\bm{x},\bm{w})$ with a set of leaf nodes $\mathcal{L}$ each described by a binary variable  $l_i$ and a prediction score  $p_i$. Splits take the form $(A_x)^\top \bm{x} + (A_w)^\top \bm{w} \leq b$, where $A_x$ gives the coefficients for the optimization variables $\bm{x}$ and $A_w$ gives the coefficients for the non-optimization (fixed) variables $\bm{w}$. Let $\mathcal{S}^l$ be the set of nodes that define the splits that observations in leaf $i$ must obey. Without loss of generality, we can write these all as $(\bar{A}_x)_j^\top \bm{x} + (\bar{A}_w)_j^\top \bm{w} - M(1-l_i) \leq \bar{b}_j$, where $\bar{A}$ is $A$ if leaf $i$ follows the left split of $j$ and $-A$ otherwise. Similarly, $\bar{b}$ equals $b$ if the leaf falls to the left split, and $-b-\epsilon$ otherwise, as established above. This decision tree can then be embedded through the following constraints: 
\begin{subequations}
\begin{align}
(\bar{A}_x)_j^\top \bm{x} + (\bar{A}_w)_j^\top \bm{w} - M(1-l_i) & \leq \bar{b}_j, \ i \in \mathcal{L}, j \in \mathcal{S}^l, \label{eqn:dt:general_start}\\ 
\sum_{i \in \mathcal{L}} l_i & = 1, \\ 
y - \sum_{i \in \mathcal{L}} p_i l_i & = 0. \label{eqn:dt:general_end}
\end{align} 
\end{subequations}
Here, $M$ can be selected for each split by considering the maximum difference between $(\bar{A}_x)_j^\top \bm{x} + (\bar{A}_w)_j^\top\bm{w}$ and $b_j$. A prescription solution $\bm{x}$ for a  patient with features $\bm{w}$ must obey the constraints determined by its split path, \textit{i.e.} only the splits that lead to its assigned leaf $i$. If $l_i = 0$ for some leaf $i$, the corresponding split constraints need not be considered. If $l_i = 1$, constraint (\ref{eqn:dt:general_start}) will enforce that the solution obeys all split constraints leading to leaf $i$. If $l_i=0$, no constraints related to leaf $i$ should be applied. When $l_i = 0$, constraint  (\ref{eqn:dt:general_start}) will be nonbinding at node $j$ if $M \geq  (\bar{A}_x)_j^\top \bm{x} + (\bar{A}_w)_j^\top\bm{w} - \bar{b}_j$. Thus we can find the minimum necessary value of $M$ by maximizing these expressions over all possible values of $\bm{x}$ (for the patient's fixed $\bm{w}$). For a given patient with features $\bm{w}$ for whom we wish to optimize treatment, EM$(\bm{w})$ is the solution of

\begin{subequations}
\begin{align}
\max_{\bm{x}} &  (\bar{A}_x)_j^\top \bm{x} + (\bar{A}_w)_j^\top\bm{w} - \bar{b}_j ~ \\
\mbox{s.t.} \ &\bm{g}(\bm{x}, \bm{w}) \leq 0, \label{eqn:bigM_g} \\
& \bm{x} \in \mathcal{X}(\bm{w}). \label{eqn:bigM_tr}
\end{align}
\end{subequations}
Note that the non-learned constraints on $\bm{x}$, namely constraint (\ref{eqn:bigM_g}), and the trust region constraint (\ref{eqn:bigM_tr}) allow us to reduce the search space when determining $M$.

\paragraph{MIO vs. LO formulation for decision trees.}\label{appendix:mio_vs_lo}
In Section~\ref{sec:methodology}, we proposed two ways of embedding a decision tree as a constraint. The first uses an LO to represent each feasible leaf node in the tree, while the second directly uses the entire MIO representation of the tree as a constraint. To compare the performance of these two approaches, we learn the palatability constraint using decision trees (CART) grown to have various numbers of leaves, and solve the optimization model with both approaches.
\begin{figure}[ht]
    \FIGURE
    {\includegraphics[scale=0.4]{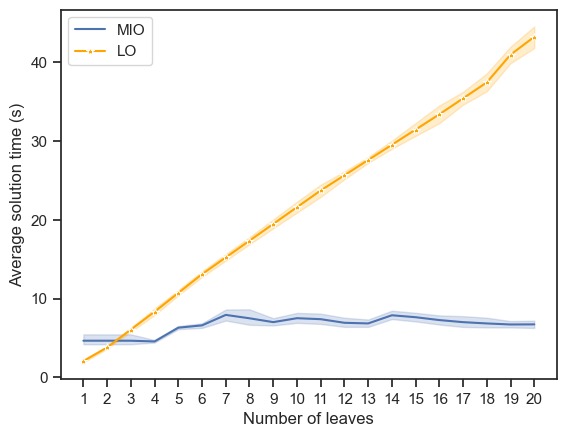}}
    {Comparison of MIO and multiple LO approach to tree representation, as a function of the number of leaves.\label{fig:MIO_VS_LP}}
    {}
\end{figure}

When comparing the solution times (averaged over 10 runs), Figure~\ref{fig:MIO_VS_LP} shows that the MIO approach is relatively consistent in terms of solution time regardless of the number of leaves. With the LO approach however, as the number of leaves grows, the number of LOs to be solved also grows. While the solution time of a single LO is very low, solving multiple LOs sequentially might be heavily time consuming. A way to speed up the process is to solve the LOs in parallel. When only one LO needs to be solved, it takes $1.8$ seconds in this problem setting. By parallelizing the solution of the LOs, the total solution time can be expected to take only as long as it takes for the slowest LO to be solved.

\subsection{Multi-layer perceptrons}\label{appendix:mlp}
MLPs consist of an input layer, $L-2$ hidden layer(s), and an output layer. In a given hidden layer $l$ of the network, with nodes $N^l$, the value of a node $i \in N^l$, denoted as $v_i^l$, is calculated using the weighted sum of the previous layer's node values, followed by the ReLU activation function, $\text{ReLU}(x) = \max\{0,x\}$. The value is given as
$$ v_i^l = \max \left\{0, \beta_{i0}^l + \sum_{j \in N^{l-1}}  \beta_{ij}^l v_j^{l-1} \right\},$$
where $\bm{\beta}_{i}^l$ is the coefficient vector for node $i$ in layer $l$. 

The ReLU operator can be encoded using linear constraints:
\begin{subequations}
\begin{align}
    v &\geq x, \label{eqn:relu_start} \\ 
    v &\leq x - M_L(1-z), \\
    v &\leq M_Uz, \\
    v &\geq 0, \\
    z &\in \left\{ 0, 1 \right\},\label{eqn:relu_end} 
\end{align}\label{eqn:relu}
\end{subequations}
where $M_L < 0$ is a lower bound on all possible values of $x$, and $M_U > 0$ is an upper bound. While this embedding relies on a big-$M$ formulation, it can be improved in multiple ways. The model can be tightened by careful selection of $M_L$ and $M_U$. Furthermore, \cite{Anderson2020} recently proposed an additional iterative cut generation procedure to improve the strength of the basic big-$M$ formulation. 

The constraints for an MLP network can be generated recursively starting from the input layer, with a set of ReLU constraints for each node in each internal layer, $l\in\left\{2,\ldots,L-1\right\}$. This allows us to embed a trained MLP with an arbitrary number of hidden layers and nodes into an MIO. 

\paragraph{Regression.}
In a regression setting, the output layer $L$ consists of a single node that is a linear combination of the node values in layer $L-1$, so it can be encoded directly as
$$y = v^L = \beta_{0}^L + \sum_{j \in N^{L-1}}  \beta_{j}^L v_j^{L-1}.$$

\paragraph{Binary Classification.}
In the binary classification setting, the output layer requires one neuron with a sigmoid activation function, $S(x) = \frac{1}{1+e^{-x}}$. The value is given as 
$$ v^L = \frac{1}{1+e^{-(\beta_0^L+\bm{\beta}^{L\top}\bm{v}^{L-1})}}$$
with $v^L \in (0,1)$. This function is nonlinear, and thus, cannot be directly embedded into our formulation. However, if $\tau$ is our desired probability lower bound, it will be satisfied when $\beta_0^L+\bm{\beta}^{L\top}\bm{v}^{L-1} \geq \ln\left(\frac{\tau}{1-\tau}\right)$. Therefore, the neural network's output, binarized with a threshold of $\tau$, is given by
\begin{align*}
    y = \begin{dcases}
    1, ~ & {\text{if }\beta_0^L+\bm{\beta}^{L\top}\bm{v}^{L-1} \geq \ln\left(\frac{\tau}{1-\tau}\right)};\\
    0, ~ & {\text{otherwise}}.
\end{dcases}
\end{align*}
For example, at a threshold of $\tau = 0.5$, the predicted value is 1 when $\beta_0^L+\bm{\beta}^{L\top}\bm{v}^{L-1} \geq 0$. Here, $\tau$ can be chosen according to the minimum necessary probability to predict 1. As for the SVC case, $y$ is binary and the constraint can be embedded as $y \geq 1$. We refer to Appendix \ref{appendix:mlp_multiclass} for the case of neural networks trained for multi-class classification.

\paragraph{Multi-class classification.}\label{appendix:mlp_multiclass}
In multi-class classification, the outputs are traditionally obtained by applying a \textit{softmax} activation function, $S(\bm{x})_i = e^{x_i}/\left(\sum_{k=1}^K{e^{x_k}}\right)$, to the final layer. This function ensures that the outputs sum to one and can thus be interpreted as probabilities. In particular, suppose we have a $K$-class classification problem. Each node in the final layer has an associated weight vector $\bm{\beta}_i$, which maps the nodes of layer $L-1$ to the output layer by $\bm{\beta}_i^\top \bm{v}^{L-1}$. The softmax function rescales these values, so that class $i$ will be assigned probability
$$v_i^L = \frac{e^{\bm{\beta_i}^\top \bm{v}^{L-1}}}{\sum_{k=1}^K{e^{\bm{\beta_k}^\top \bm{v}^{L-1}}}}.$$
We cannot apply the softmax function directly in an MIO framework with linear constraints. Instead, we use an \textit{argmax} function to directly return an indicator of the highest probability class, similar to the approach with SVC and binary classification MLP. In other words, the output $\bm{y}$ is the identity vector with $y_i = 1$ for the most likely class. Class $i$ has the highest probability if and only if
$$\beta_{i0}^L+\bm{\beta}_i^{L\top}\bm{v}^{L-1}  \geq \beta_{k0}^L+\bm{\beta}_k^{L\top}\bm{v}^{L-1}, \ \ \ k = 1, \ldots, K.$$
We can constrain this with a big-$M$ constraint as follows:
\begin{subequations}
\begin{align}
    \beta_{i0}^L+\bm{\beta}_i^{L\top}\bm{v}^{L-1}  \geq \beta_{k0}^L+\bm{\beta}_k^{L\top}\bm{v}^{L-1} - M(1-y_i), & \ \ \ k = 1, \ldots, K, \label{eqn_mc1}\\
    \sum_{k=1}^K y_k = 1.& \label{eqn_mc2}
\end{align}
\end{subequations}
Constraint~(\ref{eqn_mc1}) forces $y_i = 0$, if the constraint is not satisfied for some $k \in \{ 1, \ldots, K \}$. Constraint~(\ref{eqn_mc2}) ensures that $y_i = 1$ for the highest likelihood class. We can then constrain the prediction to fall in our desired class $i$ by enforcing $y_i = 1$.

\subsection{Model-wrapper approach}\label{app:model-wrapper}
As discussed in Section~\ref{subsect:robust_wrapper}, we can embed a set of models, rather than a single model, to improve the robustness of constraint satisfaction. This ensemble of $P$ estimators can be obtained through multiple approaches, such as through bootstrapped estimators within a single model class (e.g., $P$ linear models or $P$ decision trees) or by combining estimators across a range of model types (e.g., one linear model, one decision tree, and so on). Given the set of $P$ estimators, we then constrain that at most $\alpha$ proportion of the estimators violate the desired constraint. The $\alpha$ parameter then allows us to control the degree of conservativeness of our solution, with higher $\alpha$ values resulting in more permissive solutions and lower $\alpha$ resulting in more stringent constraint requirements. In order to constrain the violation proportion, we need indicator variables to indicate whether each of the $P$ estimators satisfies the constraint. We use a big-M formulation to obtain these indicators $z_i \ \forall i=1,\ldots,P$, as outlined in Section~\ref{subsect:robust_wrapper}. However, this does increase the complexity of the master problem through the introduction of additional binary variables. There are two special cases of the violation limit that circumvent the need for a big-M formulation:
\begin{itemize}
    \item \textbf{No allowable violation:} We can enforce a violation limit of $\alpha = 0\%$, effectively the most conservative ``worst case violation'' approach.
    \item \textbf{Average constraint:} Rather than constraining a certain proportion of estimators to obey the constraint, we can enforce that the \textit{average} prediction of all estimators obeys the constraint. This avoids the need for tracking individual constraint satisfaction for each estimator. 
\end{itemize}

In general, we note that the embedded models can be highly nonconvex on their own (e.g., if using an ensemble model as the base estimator, such as Random Forests). Thus, the additional constraints to identify and constrain violating models in this model wrapper approach are not the primary complexity drivers in the master problem, rather complexity is driven by the individual estimators. The experiments in Appendix~\ref{appendix:robustness} further investigate this latter issue: we explore runtime as the number of estimators ($P$) increases, the incremental benefit of increasing the number of estimators, and the impact of early stopping conditions.

\section{Trust region}\label{appendix:ch}
As we explain in Section~\ref{subsect:convex_hull_as_trust_region}, the trust region prevents the predictive models from extrapolating. It is defined as the convex hull of the set $\mathcal{Z} = \{ (\bm{\bar{x}}_i, \bm{\bar{w}}_i) \}_{i=1}^N$, with $\bm{\bar{x}}_i \in \mathbb{R}^n$ observed treatment decisions, and $\bm{\bar{w}}_i \in \mathbb{R}^p$ contextual information. In Section~\ref{appendix:def_convex_hull}, we explain the importance of using both $\bm{\bar{x}}$ and $\bm{\bar{w}}$ in the formulation of the convex hull. When the number of samples ($N$) is too large, the optimization model trust region constraints may become computationally expensive. In this case, we propose a column selection algorithm which is detailed in Section~\ref{appendix:columnselection}.
\subsection{Defining the convex hull}
\label{appendix:def_convex_hull}
We characterize the feasible decision space using the convex hull of our observed data. In general, we recommend defining the feasible region with respect to both $\bm{\bar{x}}$ and $\bm{\bar{w}}$. This ensures that our prescriptions are reasonable with respect to the contextual variables as well. Note that for different values of $\bm{w}$, the convex hull in the $\bm{x}$ space may be different. In Figure~\ref{fig:CHw12}, the shaded region represents the convex hull of $\mathcal{Z}$ formed by the dataset (blue dots), and the red line represents the set of trusted solutions when $\bm{w}$ is fixed to a certain value. In Figure~\ref{fig:CHw1}, we see that the set of trusted solutions (red line) lies within CH($\mathcal{Z}$) when we include $\bm{\bar{w}}$. If we leave out $\bm{\bar{w}}$ in the definition of the trust region, then we end up with the undesired situation shown in Figure~\ref{fig:CHw2}, where the solution may lie outside of CH($\mathcal{Z}$). We observe that in some cases we must define the convex hull with a subset of variables. This is true in cases where the convex hull constraint leads to excessive data thinning, in which case it may be necessary to define the convex hull on treatment variables only. 

\begin{figure}
\centering
\caption{Effect of $\bm{\bar{w}}$ on the trust region.\label{fig:CHw12}}
\begin{subfigure}[t]{0.4\textwidth}
        \includegraphics{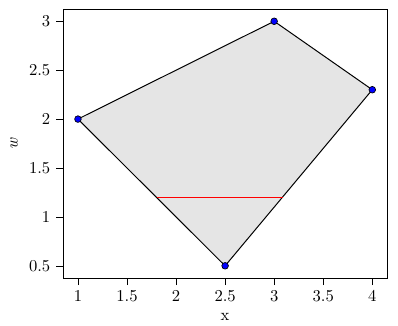}
        \caption{Solutions lie within trust region.}
        \label{fig:CHw1}
    \end{subfigure}
    \begin{subfigure}[t]{0.4\textwidth}
        \includegraphics{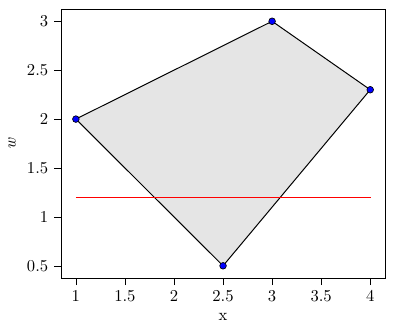}
        \caption{Solutions may lie outside the trust region.}
        \label{fig:CHw2}
    \end{subfigure}
{}
\end{figure} 

\subsection{Column selection}
\label{appendix:columnselection}
In this section, we propose a column selection method to deal with a huge set of data points. When objectives and constraints are linear, our method reduces to the Dantzig-Wolfe decomposition method \citep{Dantzig1960}. However, in our case, we do not have to solve the dual problem, since we can just enumerate all the data points. In case the functions $f$ and $g$ in formulation (\ref{eqn:conceptualmodel}) are nonlinear and convex the Dantzig-Wolfe method cannot be used. The key point in our approach is the choice of the dual problem. Although the use of Fenchel duality seems a logical way to deal with our problem, it appears that Wolfe duality, which in general leads to nonconvex formulations, is exactly what we need. In  \citep{Stoer2005} and \citep{Stoer2007} another method is described to optimize over the convex hull of a huge set of points. However, the method proposed in these papers is only suitable for problems that have only the convex hull constraint and no additional constraints.

Let P$_\mathcal{I}$ be a convex and continuously differentiable model consisting of an objective function and constraints that may be known a priori as well as learned from data. Like in Section \ref{subsect:convex_hull_as_trust_region}, we denote the index set of samples by $\mathcal{I}$. As part of the constraints, the trust region is defined on the entire set $\mathcal{Z}$. We start with the matrix $\bm{Z} \in \mathbb{R}^{N\times(n+p)}$, where each row corresponds to a given data point in  $\mathcal{Z}$. Then, model P$_\mathcal{I}$ is given as
\begin{subequations}
\begin{align}
\min_{\bm{\lambda}} \ & f(\bm{Z}^\top\bm{\lambda}) \label{eqn:general_cs1}\\
\mbox{s.t.} \ & g_j(\bm{Z}^\top\bm{\lambda}) \leq 0, \ \ \  j=1,\dots,m, & \perp \bm{\mu}, \label{eqn:general_cs2}\\
& \sum_{i \in \mathcal{I}}\lambda_i = 1, & \perp \rho, \label{eqn:general_cs3}\\
& \lambda_i \geq 0, \ \ \  i \in \mathcal{I}, & \perp \bm{\upsilon}, \label{eqn:general_cs4}
\end{align}
\end{subequations}
where the decision variable $\bm{x}$ is replaced by $\bm{Z}^\top\bm{\lambda}$. Constraints (\ref{eqn:general_cs2}) include both \textit{known} and \textit{learned} constraints, while constraints (\ref{eqn:general_cs3}) and (\ref{eqn:general_cs4}) are used for the trust region. The dual variables associated with with constraints (\ref{eqn:general_cs2}), (\ref{eqn:general_cs3}), and (\ref{eqn:general_cs4}) are $\bm{\mu} \in \mathbb{R}^m, \rho \in \mathbb{R}$, and $\bm{\upsilon} \in \mathbb{R}^N$, respectively. Note that for readability, we omit the contextual variables ($\bm{w}$) without loss of generality.

When we deal with huge datasets, solving P$_\mathcal{I}$ may be computationally expensive. Therefore, we propose an iterative column selection algorithm (Algorithm~\ref{alg:col_selection}) that can be used to speed up the optimization while still obtaining a global optima.
\begin{algorithm}
\caption {Column Selection}\label{alg:col_selection}
\begin{algorithmic}[1]
\Require{$\mathcal{I}$} \Comment{Index set of columns of $\bm{Z}^\top$}
\Ensure{$\bm{\lambda}^*$} \Comment{Optimal solution}
\State $\mathcal{I}^\prime \gets \mathcal{I}^0$ \Comment{Initial column pool} 
\While {TRUE}
    \State $\bm{\lambda}^*$, ($\bm{\mu}^*, \rho^*, \bm{\upsilon}^*$) $\gets$ P$_{\mathcal{I}^\prime}$
    \State $\bar{\mathcal{I}}$ $\gets$ \Call{WolfeDual}{$\bm{\lambda}^*$, ($\bm{\mu}^*, \rho^*, \bm{\upsilon}^*$), $\mathcal{I}^\prime, \mathcal{I}$} \Comment{Column(s) selection}
    \If{$\Bar{\mathcal{I}} \neq \emptyset$}
        \State $\mathcal{I}^\prime \gets \mathcal{I}^\prime \cup$ $\Bar{\mathcal{I}}$
    \Else
        \State Break
    \EndIf
\EndWhile
\end{algorithmic}
\end{algorithm}

The algorithm starts by initializing $\mathcal{I}^\prime \subseteq \mathcal{I}$ with an arbitrarily small subset of samples $\mathcal{I}^0$ and iteratively solves the restricted master problem  P$_\mathcal{I^\prime}$ and the \Call{WolfeDual}{} function. By solving P$_\mathcal{I^\prime}$, we get the primal and dual optimal solutions $\bm{\lambda}^*$ and ($\bm{\mu}^*, \rho^*, \bm{\upsilon}^*$), respectively. The primal and dual optimal solutions, together with $\mathcal{I}$ and $\mathcal{I}^\prime$, are given as input to \Call{WolfeDual}{} which returns a set of samples $\bar{\mathcal{I}} \subseteq \mathcal{I} \setminus \mathcal{I}^\prime$ with negative reduced cost.
If $\bar{\mathcal{I}}$ is not empty it is added to $\mathcal{I}^\prime$ and a new iteration starts, otherwise the algorithm stops, and $\bm{\lambda}^*$ (with the corresponding $\bm{x}^*$) is returned as the global optima of P$_\mathcal{I}$. A visual interpretation of Algorithm~\ref{alg:col_selection} is shown in Figure~\ref{fig:column_selection}. 

In function \Call{WolfeDual}{}, samples $\bar{\mathcal{I}}$ are selected using the  Karush–Kuhn–Tucker (KKT) stationary condition which corresponds to the equality constraint in the Wolfe dual formulation of P$_\mathcal{I}$~\citep{Wolfe_1961}. The KKT stationary condition of P$_\mathcal{I^\prime}$ is
\begin{align}
    \nabla_{\bm{\lambda}}f(\tilde{\bm{Z}}^\top\bm{\lambda}^*) + \sum_{i=1}^m \mu_i^*\nabla_{\bm{\lambda}}g_i(\tilde{\bm{Z}}^\top\bm{\lambda^*}) - \bm{e}\rho^* - \bm{\upsilon}^* = \bm{0}, \label{eqn:wolfeConstraint1}
\end{align}
where $\tilde{\bm{Z}}$ is the matrix constructed with samples in $\mathcal{I}^\prime$, and $\bm{e}$ is an $N^\prime$-dimensional vector of ones with $N^\prime = |\mathcal{I}^\prime|$.
Equation (\ref{eqn:wolfeConstraint1}) can be rewritten as
\begin{align}
    \tilde{\bm{Z}}\nabla_{\bm{x}}f(\tilde{\bm{Z}}^\top\bm{\lambda}^*) + \sum_{i=1}^m \mu_i^*\tilde{\bm{Z}}\nabla_{\bm{x}} g_i(\tilde{\bm{Z}}^\top\bm{\lambda}^*) - \bm{e}\rho^* - \bm{\upsilon}^* = \bm{0}. \label{eqn:wolfeConstraint2}
\end{align}
Equation (\ref{eqn:wolfeConstraint2}) is used to evaluate the reduced cost related to each sample $\bm{\bar{z}} \in \mathcal{Z}$ which is not in matrix $\tilde{\bm{Z}}$. Consider a new sample $\bm{\bar{z}}$ in (\ref{eqn:wolfeConstraint2}), with its associated $\lambda_{N^\prime+1}$ set equal to zero. ($\lambda_1^*,\dots,\lambda_{N^\prime}^*,\lambda_{N^\prime+1}$) is still a feasible solution of the restricted master problem P$_\mathcal{I^\prime}$, since it does not affect the value of $\bm{x}$. As a consequence, $\bm{\mu}$ and $\rho$ will not change their value, nor will $f$ and $\bm{g}$. The only unknown variable is $\upsilon_{N^\prime+1}$, namely the reduced cost of $\bm{\bar{z}}$. However, we can write it as
\begin{align}
     \begin{pmatrix}
           \bm{\upsilon}^*\\
           \upsilon_{N^\prime+1}
      \end{pmatrix} =             
      \begin{pmatrix}
           \tilde{\bm{Z}} \\
           \bm{\bar{z}}^\top
      \end{pmatrix}
      \nabla_{\bm{x}}f(\tilde{\bm{Z}}^\top\bm{\lambda^*}) + \sum_{i=1}^t \mu_i^*
      \begin{pmatrix}
           \tilde{\bm{Z}} \\
           \bm{\bar{z}}^\top
      \end{pmatrix}\nabla_{\bm{x}} g_i(\tilde{\bm{Z}}^\top\bm{\lambda^*}) - \bm{e}\rho^*. \label{eqn:wolfeConstraint3}
\end{align}
If $\upsilon_{N^\prime+1}$ is negative it means that we may improve the incumbent solution of P$_\mathcal{I^\prime}$ by including the sample $\bm{\bar{z}}$ in $\tilde{\bm{Z}}$. 
\begin{lemma}
After solving the convex and continuously differentiable problem P$_{\mathcal{I}^\prime}$, the sample in $\mathcal{I}\setminus\mathcal{I}^\prime$ with the most negative reduced cost is a vertex of the convex hull CH($\mathcal{Z}$).
\end{lemma}
\begin{proof}{Proof}
From equation (\ref{eqn:wolfeConstraint3}) we have
\begin{align}
    \upsilon_{N^\prime+1} = \bm{\bar{z}}^\top\nabla_{\bm{x}}f(\tilde{\bm{Z}}^\top\bm{\lambda^*}) + \bm{\bar{z}}^\top\nabla_{\bm{x}} \bm{g}(\tilde{\bm{Z}}^\top\bm{\lambda^*})\bm{\mu}^* - \rho^*. \label{eqn:proof1}
\end{align}
The problem of finding $\bm{\bar{z}}$, such that its reduced cost is the most negative one, can be written as a linear program where equation (\ref{eqn:proof1}) is being minimized, and a solution must lie within  CH($\mathcal{Z}$). That is,
\begin{equation}
\begin{aligned}
    \min_{\bm{z}, \bm{\lambda}} \ & \bm{z}^\top\nabla_{\bm{x}}f(\tilde{\bm{Z}}^\top\bm{\lambda^*}) + \bm{z}^\top\nabla_{\bm{x}} \bm{g}(\tilde{\bm{Z}})\bm{\mu}^* - \rho^* &\\
    \mbox{s.t.} \ & \bm{Z}^\intercal\bm{\lambda} = \bm{z},  &\\
    & \sum_{j\in \mathcal{I}}\lambda_j = 1, &\\
    & \lambda_j \geq 0, \ \ \ j \in \mathcal{I},
\label{eqn:proof2}\end{aligned}
\end{equation}
where $\bm{z}$ and $\bm{\lambda}$ are the decision variables, and $\bm{\mu}^*$, $\bm{\lambda}^*$, $\rho^*$ are fixed parameters. Since the objective function is linear with respect to  $\bm{z}$, the optimal solution of~(\ref{eqn:proof2}) will necessarily be a vertex of  CH($\mathcal{Z}$). \Halmos \end{proof}

To illustrate the benefits of column selection, consider the following convex optimization problem that we shall refer to as $P_{exp}$:
\begin{subequations}
\begin{align}
\min_{\bm{x}} \ & \bm{c}^\top \bm{x}\\
\mbox{s.t.} \ & log(\sum_{i=1}^{n}e^{x_i}) \leq t, \label{eqn:LSE_constr}\\
& \bm{A}\bm{x} \leq b, \label{eqn:leaned_constr} & \\
& \sum_{i=1}^{N}\lambda_i \bm{\bar{z}}_i = \bm{x}, \label{eqn:cs_tr1} & \\
& \sum_{j=1}^{N}\lambda_j = 1, & \\
& \lambda_j \geq 0, \ \ \ j = 1\dots N. \label{eqn:cs_tr3}
\end{align}
\end{subequations}
Without a loss of generality, we assume that the constraint (\ref{eqn:LSE_constr}) is known a priori, and constraints (\ref{eqn:leaned_constr}) are the linear embeddings of learned constraints with $\bm{A} \in \mathbb{R}^{k\times n}$ and $\bm{b}\in\mathbb{R}^k$. Constraints (\ref{eqn:cs_tr1}-\ref{eqn:cs_tr3}) define the trust region based on $N$ datapoints. Figure~\ref{fig:columns_selection_performance} shows the computation time required to solve $P_{exp}$ with different values of $n$, $k$, and $N$. The “No Column Selection” approach consists of solving $P_{exp}$ using the entire dataset. The “Column Selection” approach makes use of Algorithm~\ref{alg:col_selection} to solve the problem, starting with $|\mathcal{I}^0|=100$, and selecting only one sample at each iteration, \textit{i.e.}, the one with the most negative reduced cost. It can be seen that in all cases, the use of column selection results in significantly improved computation times. This allows us to more quickly define the trust region for problems with large amounts of data.

\begin{figure}

    \FIGURE{\includegraphics[scale=0.5]{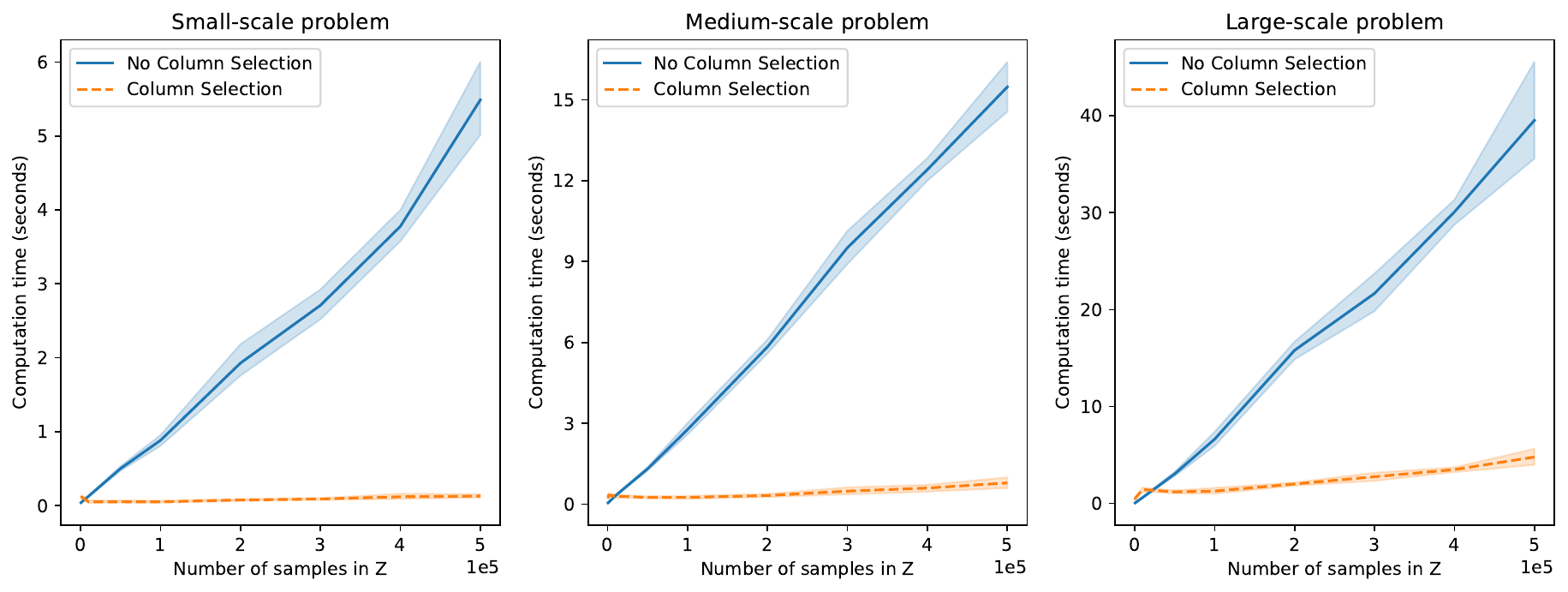}}
    {Effect of column selection on computation time. Solution times are reported for three different sizes of problem $P_{exp}$. Small-scale: $n = 5, \ k = 10$. Medium-scale: $n = 10, \ k = 50$. Large-scale: $n = 20, \ k = 100$. The number of samples goes from 500 to $5\times10^5$. In each iteration, the sample with most negative reduced cost is selected. The same problem is solved using \cite{mosek} with conic reformulation for 10 different instances where $c$, $A$, and $b$ are randomly generated. \label{fig:columns_selection_performance}\label{fig:CS_times}}
    {}
\end{figure}

\section{WFP case study}\label{appendix:wfp}
Table~\ref{tab:nutrient_values} and Table~\ref{tab:nutrient_requirements} show the nutritional value of each food and our assumed nutrient requirements, respectively. The values adopted are based on the World Health Organization (WHO) guidelines~\citep{WHOreport}.

\begin{table}[H]
\TABLE{Nutritional contents per gram for different foods.\label{tab:nutrient_values}
}{
\resizebox{\textwidth}{!}{%
\begin{tabular}{lrrrrrrrrrrrr}
\hline
\textbf{Food} &
  \textbf{Eng(kcal)} &
  \textbf{Prot(g)} &
  \textbf{Fat(g)} &
  \textbf{Cal(mg)} &
  \textbf{Iron(mg)} &
  \textbf{VitA(ug)} &
  \textbf{ThB1(mg)} &
  \textbf{RibB2(mg)} &
  \textbf{NicB3(mg)} &
  \textbf{Fol(ug)} &
  \textbf{VitC(mg)} &
  \textbf{Iod(ug)} \\ \hline
\multicolumn{1}{l|}{\textbf{Beans}}                            & 335 & 20   & 1.2 & 143  & 8.2  & 0     & 0.5  & 0.22 & 2.1 & 180 & 0  & 0       \\
\multicolumn{1}{l|}{\textbf{Bulgur}}                           & 350 & 11   & 1.5 & 23   & 7.8  & 0     & 0.3  & 0.1  & 5.5 & 38  & 0  & 0       \\
\multicolumn{1}{l|}{\textbf{Cheese}}                           & 355 & 22.5 & 28  & 630  & 0.2  & 120   & 0.03 & 0.45 & 0.2 & 0   & 0  & 0       \\
\multicolumn{1}{l|}{\textbf{Fish}}                             & 305 & 22   & 24  & 330  & 2.7  & 0     & 0.4  & 0.3  & 6.5 & 16  & 0  & 0       \\
\multicolumn{1}{l|}{\textbf{Meat}}                             & 220 & 21   & 15  & 14   & 4.1  & 0     & 0.2  & 0.23 & 3.2 & 2   & 0  & 0       \\
\multicolumn{1}{l|}{\textbf{Corn-soya blend}}            & 380 & 18   & 6   & 513  & 18.5 & 500   & 0.65 & 0.5  & 6.8 & 0   & 40 & 0       \\
\multicolumn{1}{l|}{\textbf{Dates}}                            & 245 & 2    & 0.5 & 32   & 1.2  & 0     & 0.09 & 0.1  & 2.2 & 13  & 0  & 0       \\
\multicolumn{1}{l|}{\textbf{Dried skim milk}} & 360 & 36   & 1   & 1257 & 1    & 1,500 & 0.42 & 1.55 & 1   & 50  & 0  & 0       \\
\multicolumn{1}{l|}{\textbf{Milk}}                             & 360 & 36   & 1   & 912  & 0.5  & 280   & 0.28 & 1.21 & 0.6 & 37  & 0  & 0       \\
\multicolumn{1}{l|}{\textbf{Salt}}                             & 0   & 0    & 0   & 0    & 0    & 0     & 0    & 0    & 0   & 0   & 0  & 1000000 \\
\multicolumn{1}{l|}{\textbf{Lentils}}                          & 340 & 20   & 0.6 & 51   & 9    & 0     & 0.5  & 0.25 & 2.6 & 0   & 0  & 0       \\
\multicolumn{1}{l|}{\textbf{Maize}}                            & 350 & 10   & 4   & 13   & 4.9  & 0     & 0.32 & 0.12 & 1.7 & 0   & 0  & 0       \\
\multicolumn{1}{l|}{\textbf{Maize meal}}                       & 360 & 9    & 3.5 & 10   & 2.5  & 0     & 0.3  & 0.1  & 1.8 & 0   & 0  & 0       \\
\multicolumn{1}{l|}{\textbf{Chickpeas}}                        & 335 & 22   & 1.4 & 130  & 5.2  & 0     & 0.6  & 0.19 & 3   & 100 & 0  & 0       \\
\multicolumn{1}{l|}{\textbf{Rice}}                             & 360 & 7    & 0.5 & 7    & 1.2  & 0     & 0.2  & 0.08 & 2.6 & 11  & 0  & 0       \\
\multicolumn{1}{l|}{\textbf{Sorghum/millet}}                   & 335 & 11   & 3   & 26   & 4.5  & 0     & 0.34 & 0.15 & 3.3 & 0   & 0  & 0       \\
\multicolumn{1}{l|}{\textbf{Soya-fortified bulgur wheat}}      & 350 & 17   & 1.5 & 54   & 4.7  & 0     & 0.25 & 0.13 & 4.2 & 74  & 0  & 0       \\
\multicolumn{1}{l|}{\textbf{Soya-fortified maize meal}}        & 390 & 13   & 1.5 & 178  & 4.8  & 228   & 0.7  & 0.3  & 3.1 & 0   & 0  & 0       \\
\multicolumn{1}{l|}{\textbf{Soya-fortified sorghum grits}}     & 360 & 360  & 1   & 40   & 2    & 0     & 0.2  & 0.1  & 1.7 & 50  & 0  & 0       \\
\multicolumn{1}{l|}{\textbf{Soya-fortified wheat flour}}       & 360 & 16   & 1.3 & 211  & 4.8  & 265   & 0.66 & 0.36 & 4.6 & 0   & 0  & 0       \\
\multicolumn{1}{l|}{\textbf{Sugar}}                            & 400 & 0    & 0   & 0    & 0    & 0     & 0    & 0    & 0   & 0   & 0  & 0       \\
\multicolumn{1}{l|}{\textbf{Oil}}                              & 885 & 0    & 100 & 0    & 0    & 0     & 0    & 0    & 0   & 0   & 0  & 0       \\
\multicolumn{1}{l|}{\textbf{Wheat}}                            & 330 & 12.3 & 1.5 & 36   & 4    & 0     & 0.3  & 0.07 & 5   & 51  & 0  & 0       \\
\multicolumn{1}{l|}{\textbf{Wheat flour}}                      & 350 & 11.5 & 1.5 & 29   & 3.7  & 0     & 0.28 & 0.14 & 4.5 & 0   & 0  & 0       \\
\multicolumn{1}{l|}{\textbf{Wheat-soya blend}}           & 370 & 20   & 6   & 750  & 20.8 & 498   & 1.5  & 0.6  & 9.1 & 0   & 40 & 0       \\ \hline
\end{tabular}}%
}{Eng = Energy, Prot = Protein, Cal = Calcium, VitA = Vitamin A, ThB1 = ThiamineB1, RibB2 = RiboflavinB2, NicB3 = NicacinB3, Fol = Folate, VitC = Vitamin C, Iod = Iodine}
\end{table}%

\begin{table}[H]
\TABLE{Nutrient requirements used in optimization model.\label{tab:nutrient_requirements}
}{
\resizebox{\textwidth}{!}{%
\begin{tabular}{lrrrrrrrrrrrr}
\hline
\textbf{Type} &
  \textbf{Eng(kcal)} &
  \textbf{Prot(g)} &
  \textbf{Fat(g)} &
  \textbf{Cal(mg)} &
  \textbf{Iron(mg)} &
  \textbf{VitA(ug)} &
  \textbf{ThB1(mg)} &
  \textbf{RibB2(mg)} &
  \textbf{NicB3(mg)} &
  \textbf{Fol(ug)} &
  \textbf{VitC(mg)} &
  \textbf{Iod(ug)} \\
\textbf{Avg person day} &
  2100 &
  52.5 &
  89.25 &
  1100 &
  22 &
  500 &
  0.9 &
  1.4 &
  12 &
  160 &
  0 &
  150 \\ \hline
\end{tabular}}%
}{Eng = Energy, Prot = Protein, Cal = Calcium, VitA = Vitamin A, ThB1 = ThiamineB1, RibB2 = RiboflavinB2, NicB3 = NicacinB3, Fol = Folate, VitC = Vitamin C, Iod = Iodine}
\end{table}


\subsection{Food baskets generation and palatability function}\label{app:palatability}
Referring to \cite{peters2021nutritious}, a food basket $x_k (\forall k \in \mathcal{K})$ is defined as a collection of $\mathcal{K}$ commodities, such as beans, meat, and oil, along with their respective quantities measured in grams. These commodities are classified into five macro-categories: cereals and grains, pulses and vegetables, oils and fats, mixed and blended foods, and meat and fish, as well as dairy. Each macro-category $g \in \mathcal{G}$ is associated with an upper bound ($max_g$) and a lower bound ($min_g$), as shown in Table \ref{tab:macrocat}. We use the notation $\mathcal{K}_g$ to indicate the set of commodities belonging to category $g$. In contrast to the approach used in \cite{peters2021nutritious}, where bound constraints were utilized to ensure the palatability of the food basket, we expand the notion of palatability by incorporating a palatability score that ranges from non-negative values and tends towards zero for more enjoyable diets. The score is calculated as:

\begin{align}
Palatability\ Score = \sqrt{\sum_{g\in \mathcal{G}}(\gamma_{g}(\widehat{x}_{g}-Opt_{g}))^{2}},\label{eqn:palatability_score}
\end{align}
where
$$ \widehat{x}_{g} = \sum_{k \in \mathcal{K}_g} x_{k} ~~ \text{with} \ g \in \mathcal{G} \text{ and} $$
$$ Opt_{g} = \frac{max_{g} + min_{g}}{2} ~~ \text{with} \ g \in \mathcal{G}.
$$
To account for the different range sizes ($max_g - min_g$) across the macro-categories, we introduce a scaling parameter $\gamma_{g}$ that determines their influence on the score, as presented in Table \ref{tab:macrocat}. The resulting score is normalized on a scale of 0 to 1, where a score of 1 represents a perfectly appetizing food basket, while a score of 0 indicates an inedible basket.
\begin{table}[h]
\TABLE{Macro-categories bounds and scaling factor.\label{tab:macrocat}
}{
\centering
\begin{tabular}{lccr}
\toprule
\textbf{macro-category}         & \multicolumn{1}{c}{$min$} & \multicolumn{1}{c}{\textbf{$max$}} &
\multicolumn{1}{c}{$\gamma$}\\ \midrule
Cereals \& Grains      & 200                              & 600   & 1                          \\
Pulses \& Vegetables   & 30                               & 100     &5.7                         \\
Oils \& Fats           & 15                               & 40     &16                           \\
Mixed \& Blended Foods & 0                                & 90     &4.4                           \\  
Meat \& Fish \& Dairy           & 0                               & 60     &6.6                           \\ \bottomrule
\end{tabular}}{}
\end{table}

\noindent The generation of diverse food baskets is done by solving several diet problems whose cost function changes at each run and enforcing constraints on the nutrient requirements as well as on the maximum number of foods belonging to the same category.
\subsection{Predictive models}\label{appendix:predictive_modelsWFP}
Table~\ref{tab:table_model_prameters} shows the structure of the predictive models used in the WFP experiments. For each model, the choice of parameters is based on a cross-validation procedure.

\begin{table}[htbp]
\TABLE{Definition of the predictive model parameters used in the WFP case study\label{tab:table_model_prameters}
}{
\centering
\begin{tabular}{ll}
\toprule
\textbf{Model} & \multicolumn{1}{c}{\textbf{Parameters}} \\ \midrule
Linear & ElasticNet parameters: 0.1 (alpha), 0.1 ($\ell_1$-ratio) \\
SVM & regularization parameter: 100 \\
CART & max depth: 10, max features: 1.0, min samples leaf: 0.02 \\
RF & max depth : 4, max features: auto, number of estimators: 25 \\
GBM & learning rate: 0.2, max depth: 5, number of estimators: 20 \\
MLP & hidden layers: 1, size hidden layers: (100,) activation: relu \\ \bottomrule
\end{tabular}}{}
\end{table}

\subsection{Effect of robustness parameters}\label{appendix:robustness}

\paragraph{Robustness impact by algorithm.}
Table~\ref{tab:robustness_compare} reports the change in objective value (cost) and constrained outcome (palatbility) between the nominal and bootstrapped solution with 10 estimators and a violation limit of 25\%. The goal of the WFP case study is to minimize cost such that palatability is at least 0.5; thus, a smaller cost and larger palatability are better. As expected, the robust solution increases both the cost and palatability of the prescribed diets. We see that the relative increase in cost is consistently lower than the relative increase in real palatability across all methods, indicating that the improvement in palatability exceeds the incremental cost addition. While the acceptable trade off between cost and palatability could differ by use case, this could be further explored with alternative violation limits. Additionally, we compare the single algorithm constraints against an ensemble of all six methods, also with a violation limit of $\alpha = 0.25$. The ensemble with multiple algorithms yields an objective value of 1313 and real palatability of 0.57. This represents a -1.8\% to 1\% increase in cost and 5.6\% to 15.6\% increase in real palatability over the nominal solutions. When compared to the bootstrapped single-method models, it is generally \textit{more} conservative. This is consistent with the fact that it must satisfy the constraint estimate across the majority of the individual methods, forcing it to be conservative relative to this set.

\begin{table}[htbp]
\TABLE{Change in cost and palatability from nominal to bootstrapped ($P=10, \alpha = 0.25$) solution.\label{tab:robustness_compare}
}{
\begin{tabular}{@{}lcccccc@{}}
\toprule
 & \multicolumn{3}{c}{Objective value} & \multicolumn{3}{c}{Real palatability} \\ 
\textbf{Algorithm} & \textbf{Nominal} & \textbf{Bootstrapped} & \textbf{Change} & \textbf{Nominal} & \textbf{Bootstrapped} & \textbf{Change} \\ \midrule

Linear & 1337 & 1359 & 1.6\% & 0.496 & 0.512 & 3.1\% \\
SVM & 1306 & 1308 & 0.1\% & 0.541 & 0.548 & 1.2\% \\
CART & 1301 & 1307 & 0.5\% & 0.539 & 0.550 & 2.1\% \\
RF & 1305 & 1306 & 0.0\% & 0.543 & 0.551 & 1.5\% \\
GBM & 1300 & 1304 & 0.3\% & 0.532 & 0.553 & 3.9\% \\
MLP & 1307 & 1313 & 0.5\% & 0.537 & 0.587 & 9.4\% \\
\bottomrule
\end{tabular}}{}
\end{table}

\paragraph{Effect of number of estimators.}

Table~\ref{tab:robustness_time} compares the runtime as the number of estimators ($P$) increases up to 25 estimators. We see that the solve time for the linear, SVM, CART, and MLP models are stable as the number of estimators increases. In contrast, we see that the ensemble algorithms, RF and GBM, have exponential runtime increases as the number of estimators grows. RF and GBM are already comprised of multiple individual learners, so embedding multiple estimators involves adding multiple \textit{sets} of decision trees, which becomes computationally expensive. All results are reported over 100 instances. The experiments were run using a virtual computing environment with 4 CPU and 32 GB total RAM. We also report the runtime for an ensemble of estimators obtained from different model classes (``Ensemble"), using a single model from each class.

We further investigate the runtimes with 25 estimators in Table~\ref{tab:robustness_25}. The left side of the table reports the mean, median, and maximum runtimes for each method on the same 100 experiments as above. We see that the RF and GBM models have reasonable median solve times (6.66 and 18.80 minutes, respectively), but the average solve times are driven up by outlier instances that have significantly higher runtimes (max. 2110 and 1603 minutes, respectively). We propose to use a time limit to control the experiment times. On the right side of the table, we see that using a 4 hour time limit returns optimal solutions for 95\% of the RF runs and 82\% of the GBM runs, and feasible solutions for all but four GBM instances. In cases where an optimal solution is obtained, the average runtime is less than 40 minutes. In cases where the time limit is hit, the average remaining MIP gap is 1.02\% for RF and 5.21\% for GBM. The results suggest that imposing this termination condition results in high quality solutions with a modest optimality gap.

\begin{table}[htbp]
\TABLE{Optimization solver runtime (minutes) as a function of number of bootstrapped estimators. \label{tab:robustness_time}
}{
\begin{tabular}{@{}lcccc@{}}
\toprule
\textbf{Algorithm} & \textbf{P = 2} & \textbf{P = 5} & \textbf{P = 10} & \textbf{P = 25} \\ \toprule
\textbf{Linear} & 0.01           & 0.01           & 0.02            & 0.01            \\
\textbf{SVM}    & 0.01           & 0.01           & 0.02            & 0.01            \\
\textbf{CART}   & 0.02           & 0.02           & 0.03            & 0.12            \\
\textbf{RF}     & 0.15           & 1.34           & 11.93           & 44.87           \\
\textbf{GBM}    & 0.37           & 3.58           & 10.71           & 133.13          \\
\textbf{MLP}    & 0.01           & 0.02           & 0.04            & 0.48            \\ \midrule
\textbf{Ensemble} & \multicolumn{4}{c}{0.09} \\ \bottomrule
\end{tabular}}{}
\end{table}

\begin{table}[htbp]
\TABLE{Runtime results for $P=25$ estimators, both when solved to optimality (left) and with a 4 hour time limit (right). \label{tab:robustness_25}
}{
\resizebox{\textwidth}{!}{%
\begin{tabular}{@{}l|C{1.5cm}C{1.5cm}C{1.5cm}|cccc@{}}
\toprule
{ } & \multicolumn{3}{c|}{\textbf{Runtime (mins) to optimality}} & \multicolumn{4}{c}{\textbf{4 hour time limit}} \\ \midrule
\textbf{Algorithm} & \textbf{Mean} & \textbf{Median} & \textbf{Max} & \textbf{\% feasible} & \textbf{\% optimal} & \textbf{\begin{tabular}[c]{@{}c@{}}Avg. runtime \\ to optimality \\ (mins) \end{tabular}} & \textbf{\begin{tabular}[c]{@{}c@{}}Avg. \\ remaining \\ MIP gap\end{tabular}} \\
\textbf{Linear} & 0.01 & 0.01 & 0.04 & 100\% & 100\% & 0.01 & -- \\
\textbf{SVM} & 0.01 & 0.01 & 0.04 & 100\% & 100\% & 0.01 & -- \\
\textbf{CART} & 0.12 & 0.08 & 0.64 & 100\% & 100\% & 0.12 & -- \\
\textbf{RF} & 44.87 & 6.66 & 2109.72 & 100\% & 95\% & 19.76 & 1.02\% \\
\textbf{GBM} & 133.13 & 18.80 & 1603.51 & 96\% & 82\% & 37.79 & 5.21\% \\
\textbf{MLP} & 0.49 & 0.12 & 6.89 & 100\% & 100\% & 0.48 & -- \\ \bottomrule
\end{tabular}}}
{}
\end{table}


The runtime experiments raise a natural question: what is the impact of embedding a larger number of estimators? We consider the cost-palatability trade off for a decision tree model as we vary the number of estimators from $P=2$ to $P=50$, averaged over the candidate violation limits. The results are shown in Figure~\ref{fig:robustness_P}. As the number of estimators increases, the results tend to be more conservative. By 10 estimators, the trade off curve well-approximates the curves for higher estimator up to an inflection point where average cost increases significantly. By $P=25$ and $P=50$, the curves closely match, suggesting diminishing value in increasing the number of estimators beyond a certain point.

Finally, Table~\ref{tab:robustness_params} reports the parameters used in our bootstrapped models. For each method, we report the parameter grid that was used in our model training and selection procedure. Individual estimators use different combinations of these parameters based on the validation performance on the specific bootstrapped samples. We note that for these experiments, we used a default parameter grid implemented in \texttt{OptiCL}; this grid can be manually set by a user when specifying each outcome of interest before model training.

\begin{table}[htbp]
\TABLE{Default parameter grid for supported algorithms. \label{tab:robustness_params}
}{
\begin{tabular}{@{}ll@{}}
\toprule
\textbf{Algorithm} & \textbf{Parameter Grid} \\ \midrule
\textbf{Linear} & \begin{tabular}[c]{@{}l@{}}alpha': {[}0.1, 1, 10, 100, 1000{]}, \\ 'l1\_ratio': np.arange(0.1, 1.0, 0.2)\end{tabular} \\ \midrule
\textbf{SVM} & 'C': {[}.1, 1, 10, 100{]} \\ \midrule
\textbf{CART} & \begin{tabular}[c]{@{}l@{}}max\_depth': {[}3, 4, 5, 6, 7, 8, 9, 10{]}, \\ 'min\_samples\_leaf': {[}0.02, 0.04, 0.06{]}, \\ 'max\_features': {[}0.4, 0.6, 0.8, 1.0{]}\end{tabular} \\ \midrule
\textbf{RF} & \begin{tabular}[c]{@{}l@{}}n\_estimators': {[}10, 25{]}, \\ 'max\_features': {[}'auto'{]}, \\ 'max\_depth': {[}2, 3, 4{]}\end{tabular} \\ \midrule
\textbf{GBM} & \begin{tabular}[c]{@{}l@{}}learning\_rate': {[}0.01, 0.025, 0.05, 0.075, 0.1, 0.15, 0.2{]}, \\ 'max\_depth': {[}2, 3, 4, 5{]}, \\ 'n\_estimators': {[}20{]}\end{tabular} \\ \midrule
\textbf{MLP} & 'hidden\_layer\_sizes': {[}(10,), (20,), (50,), (100,){]} \\ \bottomrule
\end{tabular}}
{}
\end{table}

\begin{figure}[htbp]
\FIGURE{
    \centering
    \includegraphics[width=\linewidth]{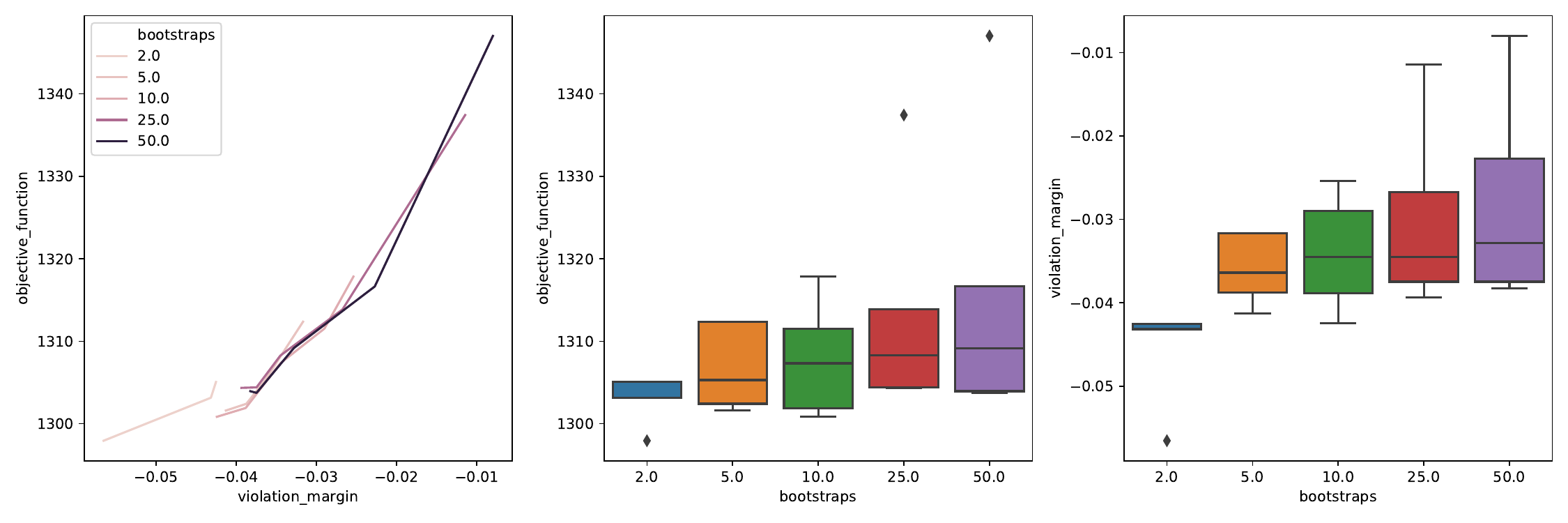}
}{Effect of the number of bootstrapped estimators ($P$) on the cost and palatability of the prescribed diet.\label{fig:robustness_P}}{}
\end{figure}


\section{Chemotherapy regimen design}
\subsection{Data Processing}\label{appendix:gastric:data}
The data for this case study includes three components, study cohort characteristics ($\bm{w}$), treatment variables ($\bm{x}$), and outcomes ($\bm{y}$). The raw data was obtained from ~\citet{Bertsimas2016}, in which the authors manually curated data from 495 clinical trial arms for advanced gastric cancer. Our feature space was processed as follows:

\paragraph{Cohort Characteristics.}
We included several cohort characteristics to adjust for the study context: fraction of male patients, median age, primary site breakdown (Stomach vs. GEJ), fraction of patients receiving prior palliative chemotherapy, and mean ECOG score. We also included variables for the study context: the study year, country, and number of patients. Missing data was imputed using multiple imputation based on the other contextual variables; 20\% of observations had one missing feature and 6\% had multiple missing features.

\paragraph{Treatment Variables.}
Chemotherapy regimens involve multiple drugs being delivered at potentially varied frequencies over the course of a chemotherapy cycle. As a result, multiple dimensions of the dosage must be encoded to reflect the treatment strategy. As in ~\citet{Bertsimas2016}, we include three variables to represent each drug: an indicator (1 if the drug is used in the regimen), instantaneous dose, and average dose. 

\paragraph{Outcomes.}
We use Overall Survival (OS) as our survival metric, as reported in the clinical trials. Any observations with unreported OS are excluded. We consider several ``dose-limiting toxicities" (DLTs): Grade 3/4 constitutional, gastrointestinal, infection, and neurological toxicities, as well as Grade 4 blood toxicities. The toxicities reported in the original clinical trials are aggregated according to the CTCAE toxicity classes~\citep{ctcae}. We also include a variable for the occurrence of any of the four individual toxicities ($t_i$ for each toxicity $i \in T$, called DLT proportion; we treat these toxicity groups as independent and thus define the DLT proportion as
\begin{align*}
    DLT = 1 - \prod_{i \in T} (1-t_i).
\end{align*}
We define Grade 4 blood toxicity as the maximum of five individual blood toxicities (related to neutrophils, leukocytes, lymphocytes, thrombocytes, anemia). Observations missing all of these toxicities were excluded; entries with partial missingness were imputed using multiple imputation based on other blood toxicity columns. Similarly, observations with no reported Grade 3/4 toxicities were excluded; those with partial missingness were imputed using multiple imputation based on the other toxicity columns. This exclusion criteria resulted in a final set of 461 (of 495) treatment arms.

We split the data into training/testing sets temporally. The training set consists of all clinical trials through 2008, and the testing set consists of all 2009-2012 trials. We exclude trials from the testing set if they use new drugs not seen in the training data (since we cannot evaluate these given treatments). We also identify sparse treatments (defined as being only seen once in the training set) and remove all observations that include these treatments. The final training set consists of 320 observations, and the final testing set consists of 96 observations. 

\subsection{Predictive Models}\label{appendix:gastric:predictive}
Table~\ref{tab:gastric:predictive_all} shows the out-of-sample performance of all considered methods in the model selection pipeline. We note that model choice is based on the 5-fold validation performance, so it does not necessarily correspond to the highest test set performance. The final parameters for each model and each outcome, selected through the cross-validation procedure, are shown in Table~\ref{tab:gastric:predictive_all_params}.


\begin{table}
\TABLE
    {Comparison of out-of-sample $R^2$ all considered models for learned outcomes in chemotherapy regimen selection problem.}
    {\begin{tabular}{@{}lrrrrrr@{}}
\toprule
\textbf{Outcome} & \textbf{Linear} & \textbf{SVM} & \textbf{CART} & \textbf{RF} & \textbf{GBM} \\ \midrule
Any DLT          & 0.268 & -0.094 & -0.016  & 0.152 & 0.202 \\
Blood            & 0.196 & -1.102 & 0.012  & 0.153 & 0.105 \\
Constitutional   & 0.106 & 0.144  & 0.157  & 0.194 & 0.136 \\
Infection        & 0.082 & -0.511 & -0.222  & 0.070  & 0.035 \\
Gastrointestinal & 0.141 & -0.196 & -0.023 & 0.066 & 0.083 \\ \midrule
Overall Survival & 0.448 & 0.385  & 0.474  & 0.496 & 0.450 \\ \bottomrule
\end{tabular}}
 {\label{tab:gastric:predictive_all}}
\end{table}

\begin{table}[htbp]
\TABLE{Predictive model parameters used in the chemotherapy case study.\label{tab:gastric:predictive_all_params}
}{
\centering
\begin{tabular}{lll}
\toprule
\textbf{Outcome} & \textbf{Model} & \multicolumn{1}{c}{\textbf{Parameters}} \\ \midrule
Any DLT & GBM & learning rate: 0.2, max depth: 2, number of estimators: 20 \\
Blood & Linear & ElasticNet parameters: 0.1 (alpha), 0.7 ($\ell_1$-ratio) \\
Constitutional & RF &  max depth : 4, max features: 'auto', number of estimators: 25 \\
Infection & Linear & ElasticNet parameters: 1 (alpha), 0.5 ($\ell_1$-ratio) \\
Gastrointestinal & GBM & learning rate: 0.1, max depth: 4, number of estimators: 20 \\
\midrule
Overall Survival & GBM & learning rate: 0.1, max depth: 3, number of estimators: 20 \\
\bottomrule
\end{tabular}}{}
\end{table}

\subsection{Prescription Evaluation}\label{appendix:gastric:evaluation}

Table~\ref{tab:gastric:predictive_gt} shows the performance of the models that comprise the ground truth ensemble used in the evaluation framework. These models trained on the full data. We see that the ensemble models, particularly RF and GBM, have the highest performance. These models are trained on more data and include more complex parameter options (\textit{e.g.}, deeper trees, larger forests) since they are not required to be embedded in the MIO and are rather used directly to generate predictions. The final parameters for each model and each outcome, selected through the cross-validation procedure, are shown in Table~\ref{tab:gastric:predictive_gt_params}. For this reason, the GT ensemble could also be generalized to consider even broader method classes that are not directly MIO-representable, such as neural networks with alternative activation functions, providing an additional degree of robustness.

\begin{table}
\TABLE
    {Performance ($R^2$) of individual models in ground truth ensemble for model evaluation.}
    {\begin{tabular}{@{}lrrrrrr@{}}
\toprule
\textbf{outcome} & \textbf{Linear} & \textbf{SVM} & \textbf{CART} & \textbf{RF} & \textbf{GBM} & \textbf{XGB} \\ \midrule
Any DLT          & 0.301 & 0.330  & 0.250  & 0.573 & 0.670  & 0.323 \\
Blood            & 0.287 & 0.351 & 0.211 & 0.701 & 0.813 & 0.446 \\
Constitutional   & 0.139 & 0.224 & 0.246 & 0.602 & 0.682 & 0.285 \\
Infection        & 0.217 & 0.303 & 0.139 & 0.514 & 0.588 & 0.247 \\
Gastrointestinal & 0.201 & 0.328 & 0.238 & 0.563 & 0.733 & 0.475 \\ \midrule
Overall Survival & 0.528 & 0.469 & 0.421 & 0.815 & 0.827 & 0.756 \\ \bottomrule
\end{tabular}}
 {\label{tab:gastric:predictive_gt}}
\end{table}

\begin{table}[htbp]
\TABLE{Predictive model parameters used in the ground truth ensemble for model evaluation. \label{tab:gastric:predictive_gt_params}
}{
\centering
\begin{tabular}{@{}llllllll@{}}
\toprule
\textbf{Algorithm} & \textbf{Parameter} & \textbf{Any DLT} & \textbf{Blood} & \textbf{Const.} & \textbf{Inf.} & \textbf{GI} & \textbf{OS} \\ \midrule
\textit{Linear} & alpha & 0.1 & 0.1 & 1 & 1 & 1 & 0.1 \\
 & $\ell_1$ ratio & 0.6 & 0.5 & 0.4 & 0.3 & 0.7 & 0.8 \\
 \midrule
\textit{SVM} & regularization parameter & 100 & 100 & 1 & 10 & 100 & 0.1 \\
\midrule
\textit{CART} & max depth & 3 & 3 & 4 & 3 & 5 & 3 \\
 & max features & 1 & 1 & 0.6 & 0.6 & 0.6 & 0.8 \\
 & min samples per leaf & 0.04 & 0.06 & 0.06 & 0.06 & 0.06 & 0.02 \\
 \midrule
\textit{RF} & max depth & 6 & 8 & 6 & 6 & 6 & 8 \\
 & max features & auto & auto & auto & auto & auto & auto \\
 & number of estimators & 500 & 500 & 500 & 250 & 250 & 250 \\
 \midrule
\textit{GBM} & learning rate & 0.01 & 0.025 & 0.01 & 6 & 0.01 & 0.01 \\
 & max depth & 5 & 5 & 5 & auto & 6 & 5 \\
 & number of estimators & 250 & 250 & 250 & 250 & 250 & 250 \\
 \midrule
\textit{XGB} & cols sampled by tree & 0.8 & 1 & 0.8 & 1 & 0.8 & 1 \\
 & gamma & 0.5 & 0.5 & 1 & 1 & 0.5 & 10 \\
 & max depth & 4 & 5 & 4 & 4 & 5 & 4 \\
 & min child weight & 10 & 1 & 10 & 10 & 1 & 10 \\
 & number of estimators & 250 & 250 & 250 & 250 & 250 & 250 \\
 & subsample & 1 & 0.8 & 0.8 & 0.8 & 0.8 & 1 \\ \bottomrule
\end{tabular}}{
Const. = Constitutional, Inf. = Infection, GI = Gastrointestinal, OS = Overall survival.
}
\end{table}
\subsection{Optimization runtimes}\label{appendix:gastric:runtime}
Table~\ref{tab:gastric:runtimes} reports the runtimes of the optimization model results presented in Section~\ref{subsubsect:gastric:optimization_results}, Table~\ref{tab:gastric:gt}. Results are averaged over all patients in the test set.
\begin{table}[h]
\TABLE{Average (and standard deviation) of runtimes for gastric cancer case, in seconds. \label{tab:gastric:runtimes}
}{
    \begin{tabular}{lc}
    \toprule
    \textbf{Model Version} & \textbf{Average Time (SD)} \\
    \midrule
    All Constraints   &  0.511 (0.892) \\
    DLT Only  & 0.203 (0.433) \\
    \bottomrule
    \end{tabular}
    }{}
\end{table}

\section{Comparison with JANOS and EML}
\label{appendix:opticlvjanos}
As mentioned earlier in Section \ref{subsect:litrev}, JANOS and EML are two software frameworks for embedding learned ML models in optimization problems. In this section, we compare the performance of \texttt{OptiCL} to those of JANOS and EML using the test problems in \citet{bergman2019janos} and \citet{lombardi2017empirical}, respectively. The experiments are conducted using an Intel i7-8665U 1.9 GHz CPU, 16 GB RAM (Windows 10 environment).

\subsection{OptiCL vs JANOS}\label{appendix:opticlvJANOS}
In the Student Enrolment Problem (SEP) in \citet{bergman2019janos}, a university's admission office seeks to offer scholarships to some of the admitted students in order to bolster the class profile. The objective is to maximize the expected class size subject to budget constraints. This problem is formulated as:

\begin{subequations}
\begin{align}
\max & \ \sum_{i=1}^{N} y_{i}\\
\mbox{s.t.} \ & \sum_{i=1}^{N} x_i \leq \textbf{BUDGET},  \\
& y_i = \hat{h}(s_i, g_i, x_i) & \forall i \in \{ 1, \ldots, N\}, \\
& 0 \leq x_i \leq 25,000 & \forall i \in \{ 1, \ldots, N\}, 
\end{align}
\end{subequations}

\noindent
where $x_i$ is the decision variable indicating the amount of scholarship assigned to each student accepted, $s_i$ is the SAT score of applicant $i$, and $g_i$ is the GPA score of applicant $i$. The predicted outcome $y_i$ represents the probability of a candidate $i$ accepting the offer, and $\hat{h}$ is the fitted model used to predict any candidate's probabilities of accepting an offer. The parameters $s_i$, $g_i$, and the decision variable $x_i$ are the predictive model’s inputs.
In order to compare OptiCL and JANOS, we solved the SEP for different student sizes, and compared the objective values and runtimes. Although \texttt{OptiCL} and JANOS handle neural network embedding in a similar manner, JANOS uses a parameterized discretization to handle logistic regression predictions. We therefore compared their performances only using the logistic regression models, as we expected to see a difference in performance based on the differences in implementation.  In the experiments reported in Figure~\ref{fig:OPTICLvsJANOS}, we discretize the logistic regression (LogReg) in JANOS using three different number of intervals (reported between brackets in the Figure legend). From the experiments, we can see that \texttt{OptiCL} achieves better objective values in all three instances. It can also be seen that for the larger problems, \texttt{OptiCL} is much more efficient in terms of optimization runtime than JANOS.

\begin{figure}[ht]
    \caption{Objective value (right) and runtime (left) comparison between \texttt{OptiCL and JANOS} for the SEP.}%
    \label{fig:OPTICLvsJANOS}%
    \centering{{\includegraphics[width=5.4cm]{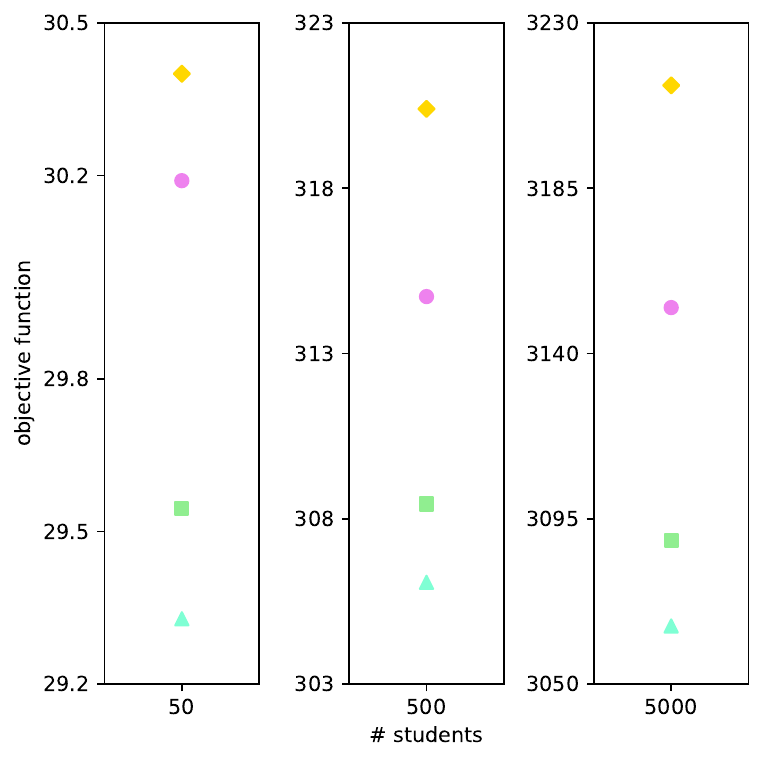} }}%
    \qquad{{\includegraphics[width=9.5cm]{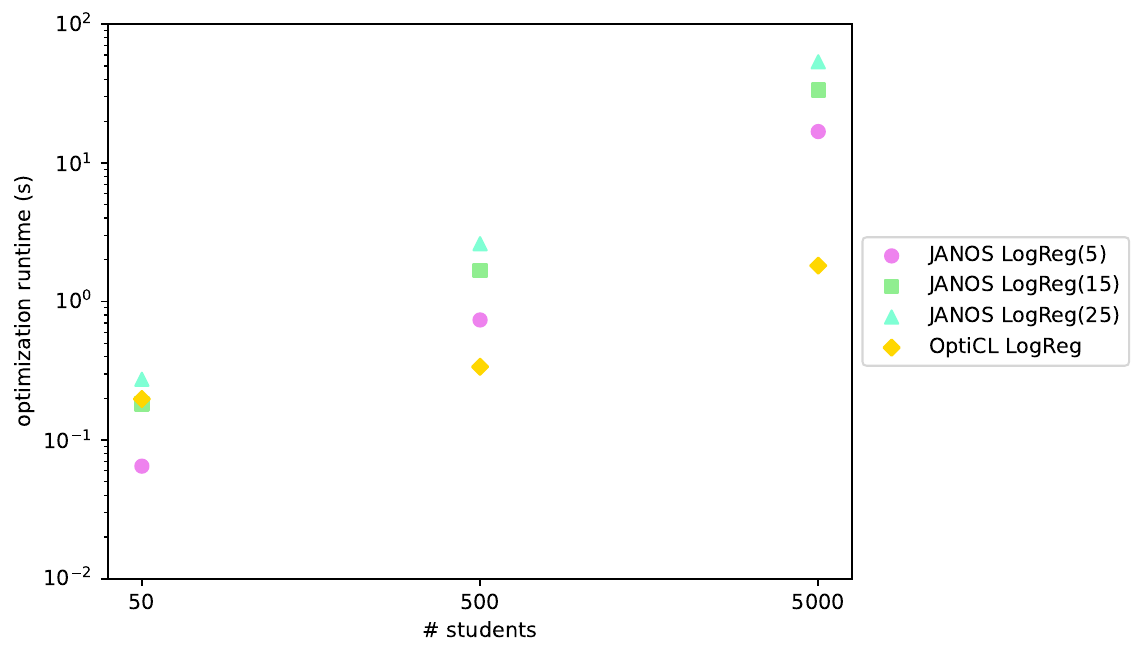} }}%
\end{figure}

\subsection{OptiCL vs EML}\label{appendix:opticlvEML}
In the thermal-aware Workload Dispatching Problem (WDP) in \citet{lombardi2017empirical}, the goal is to assign jobs to the different cores on a multi-core processor. The processor has 24 dual-core tiles arranged in a 4$\times$6 grid, resulting in an arrangement with 48 cores in an 8$\times$6 grid. A direct comparison between \texttt{OptiCL} and EML is not possible, as \citet{lombardi2017empirical} do not use neural networks or decision trees for constraint learning in MIO problems. Their focus for these predictive models are Local Search, Constraint Programming, or SAT Modulo Theory problems. What we do, however, is demonstrate that \texttt{OptiCL} is able to solve the example in an MIO setting. The model considered here is the ``ANN1" model in \citet{lombardi2017empirical} given as:

\begin{subequations}
\begin{align}
\max & \ z \\
\mbox{s.t.} \ & z \leq y_{k} & \forall k = 0, \ldots, m-1,  \label{eqn:zleqy}\\
& y_{k} = \hat{h}_{k}(avgcpi_k, neighcpi_k, othercpi_k) & \forall k = 0, \ldots, m-1,  \\
& \sum_{k=0}^{m-1} x_{ik} = 1 & \forall i = 0, \ldots, n-1,  \label{eqn:onejob}\\
& \sum_{i=0}^{n-1} x_{ik} = \dfrac{n}{m} & \forall k = 0, \ldots, m-1,  \label{eqn:equalnumjobs}\\
& avgcpi_k = \dfrac{1}{m}\sum_{i=0}^{n-1}cpi_ix_{ik} & \forall k = 0, \ldots, m-1 \label{eqn:avgcpi},  \\
& neighcpi_k = \dfrac{1}{m}\sum_{h\in N(k)}avgcpi_h & \forall k = 0, \ldots, m-1, \label{eqn:neighcpi} \\
& othercpi_k = \dfrac{1}{m-1-|N(k)|}\sum_{h\neq k,h\notin N(k)}avgcpi_h & \forall k = 0, \ldots, m-1, \label{eqn:othercpi} \\
& x_{ik} \in \{0,1\}  & \forall i = 0, \ldots, n-1 ~~~  \forall k = 0, \ldots, m-1,
\end{align}
\end{subequations}

\noindent
where $x_{ik}$ is the binary decision variable indicating if a job $i$ is mapped on core $k$ or not. The parameter $cpi_i$ represents the average Clock Per Instructions (CPI) characterizing job $i$, and is a measure of the difficulty of job $i$. The objective is to maximize the worst-case core efficiency, and the fitted model $\hat{h}_k$ is used to predict the efficiency of core $k$ that is represented by $y_k \in [0,1]$. 

Constraints \eqref{eqn:onejob} ensures that each job is mapped to only one core, and \eqref{eqn:equalnumjobs} forces the same number of jobs to run on each core. Constraints \eqref{eqn:avgcpi}, \eqref{eqn:neighcpi} and \eqref{eqn:othercpi} are used to compute the average CPI for a core $k$, the average CPI for the cores in the neighborhood of $k$ ($N(k)$), and the average CPI for cores not in the neighborhood of $k$ respectively.
\cite{lombardi2017empirical} conclude that learning the efficiency function for each core by means of neural networks (with one hidden layer of two nodes and $\tanh$ activation function) is computationally intractable. On the contrary, our experiments show that we are able to solve this problem using neural networks with one hidden layer and 10 nodes in a reasonable amount of time (19.4 seconds). We tried deeper neural networks, but the increase in computational complexity did not lead to a gain in predictive performance.

\end{APPENDICES}
\end{document}